\RequirePackage{fix-cm}
\documentclass[a4paper,12pt]{article}
\usepackage{amsmath,amsthm}
\usepackage{amssymb}
\usepackage{latexsym}
\usepackage{graphicx}
\usepackage{tikz}
\usepackage{scalefnt}
\usepackage{lmodern}
\usetikzlibrary{positioning}
\usetikzlibrary{shapes,arrows}
\usepackage{subfigure}
\usepackage{pgfplots}
\usepackage{caption}
\usepackage[pdfpagemode=UseNone,pdfstartview=FitH,hypertexnames=false]{hyperref}
\newtheorem{theorem}{Theorem}[section]
\newtheorem{corollary}{Corollary}[section]
\newtheorem{remark}{Remark}[section]

\newtheorem{definition}{Definition}[section]

\begin{document}
\title{{\Large \bf A lattice-theoretic approach to the Bourque-Ligh conjecture}}
\author{{\sc Ismo Korkee, Mika Mattila and Pentti Haukkanen}
\\School of Information Sciences
\\FI-33014 University of Tampere, Finland
\\E-mail:  ismo.korkee@taokk.tamk.fi, mika.mattila@uta.fi,\\pentti.haukkanen@uta.fi}
\date{August 31, 2013}
\maketitle
\setcounter{section}{0}
\setcounter{equation}{0}
\noindent{\bf Abstract } The Bourque-Ligh conjecture states that if $S=\{x_1,x_2,\ldots,x_n\}$ is a gcd-closed set of positive integers with distinct elements, then the LCM matrix $[S]=[\hbox{lcm}(x_i,x_j)]$ is invertible. It is well known that this conjecture holds for $n\leq7$ but does not generally hold for $n\geq8$. In this paper we provide a lattice-theoretic explanation for this solution of the Bourque-Ligh conjecture. In fact, let $(P,\leq)=(P,\land,\lor)$ be a lattice, let $S=\{x_1,x_2,\ldots,x_n\}$ be a subset of $P$ and let $f:P\to{\mathbb C}$ be a function. We study under which conditions the join matrix $[S]_f=[f(x_i\lor x_j)]$ on $S$ with respect to $f$ is invertible on a meet closed set $S$ (i.e.,~$x_i,x_j\in S\Rightarrow x_i\land x_j\in S)$.


\section{Introduction}
\label{sect:1}
\setcounter{equation}{0}

Let $(P,\leq)=(P,\land,\lor)$ be a lattice, let $S=\{x_1,x_2,\ldots,x_n\}$ be a subset of $P$ and let $f:P\to{\mathbb C}$ be a function. The meet matrix $(S)_f$ and the join matrix $[S]_f$ on $S$ with respect to $f$ are defined by $((S)_f)_{ij}=f(x_i\land x_j)$ and $([S]_f)_{ij}=f(x_i\lor x_j)$. Rajarama Bhat \cite{Raja} and Haukkanen \cite{Hau96} introduced meet matrices and Kor\-kee and Haukkanen \cite{KorHau2} defined join matrices. Explicit formulae for the determinant and the inverse of meet and join matrices are presented in \cite{Hau96,KorHau1,KorHau2,Raja} (see also \cite{AltSagTug,HongSun}). Most of these formulae are presented on meet closed sets $S$ (i.e.,~$x_i,x_j\in S\Rightarrow x_i\land x_j\in S)$ and join-closed sets $S$ (i.e.,~$x_i,x_j\in S\Rightarrow x_i\lor x_j\in S)$. More recently Kor\-kee and Haukkanen \cite{KorHau6} presented a method for calculating $\det(S)_f$, $(S)_f^{-1}$, $\det[S]_f$ and $[S]_f^{-1}$ on all sets $S$ and functions $f$. It is well known that $({\mathbb Z}_+,\mid)=({\mathbb Z}_+,\hbox{gcd},\hbox{lcm})$ is a lattice, where $\mid$ is the usual divisibility relation and gcd and lcm stand for the greatest common divisor and the least common multiple of integers. Thus meet and join matrices are generalizations of GCD matrices ${((S)_f)_{ij}}={f(\hbox{gcd}(x_i,x_j))}$ and LCM matrices $([S]_f)_{ij}=f(\hbox{lcm}(x_i,x_j))$, where $f$ is an arithmetical function. If $f=N$, where $N(m)=m$ for all positive integers $m$, then we denote $(S)_f=(S)$ and $[S]_f=[S]$. The study of GCD and LCM matrices is considered to have begun in 1876, when Smith \cite{Smi} presented his famous determinant formulae. The GCUD and LCUM matrices, which are unitary analogues of GCD and LCM matrices, are also special cases of meet and join matrices, see \cite{HauIlNalSil,HauSil96,Kor8}. For general accounts of meet and join matrices and their number-theoretic special cases, see \cite{HauWanSil,KorHau2,Sandor}.

Bourque and Ligh \cite{Bour92} conjectured that the LCM matrix $[S]$ on any gcd-closed set is invertible. Haukkanen, Wang and Sillanp\"a\"a \cite{HauWanSil} were the first to show that the conjecture does not hold (giving a counterexample with $n=9$). Hong \cite{Hong99} solved the conjecture completely in the sense that it holds for $n\leq7$ and does not hold generally for $n\geq8$. Subsequently he also presented some conjectures on his own \cite{Hong281,HongRfold,Li}.

In this paper we study a lattice-theoretic generalization of the Bourque-Ligh conjecture, i.e.,~under which conditions the join matrix $[S]_f$ is invertible on a meet closed set $S$. We use the concept of covering to develop an inductive method for inserting an element to $S$ so that the invertibility of the join matrix on the extended set is preserved. We apply this method to explain in terms of lattice theory why $n=7$ is the greatest integer for which the original Bourque-Ligh conjecture holds.

\section{Preliminaries}
\label{sect:2}
\setcounter{equation}{0}

Let $(P,\leq)$ be a locally finite poset and let $g$ be an incidence function of $P$, that is, $g$ is a complex-valued function on $P\times P$ such that $g(x,y)=0$ whenever $x\not\leq y$. If $h$ is also an incidence function of $P$, the sum $g+h$ is defined by $(g+h)(x,y)=g(x,y)+h(x,y)$ and the convolution $g*h$ is defined by ${(g*h)(x,y)}=\sum_{x\leq z\leq y}g(x,z)h(z,y)$. The set of all incidence functions of $P$ under addition and convolution forms a ring with unity, where the unity $\delta$ is defined by $\delta(x,y)=1$ if $x=y$, and $\delta(x,y)=0$ otherwise. The zeta incidence function $\zeta$ is defined by $\zeta(x,y)=1$ if $x\leq y$, and $\zeta(x,y)=0$ otherwise. The M\"obius function $\mu$ of $P$ is the inverse of $\zeta$ (with respect to the convolution).

In this paper let $(P,\leq)=(P,\land,\lor)$ always be a lattice such that the principal order ideal $\downarrow\hspace{-1.5mm}x={\{y\in P\ \vert\ y\leq x\}}$ is finite for each $x\in P$. Then $P$ has the least element, which we denote by $0$. The order ideal generated by $S$ is ${\downarrow\hspace{-1mm} S}={\{z\in P\mid \exists x\in S:z\leq x\}}$, see \cite{Bir}. Let $f$ always be a complex-valued function on $P$ and let $S$ be a finite subset of $P$, where $S=\{x_1,x_2,\ldots,x_n\}$ with $x_i<x_j\Rightarrow i<j$. We say that $S$ is an $a$-set if $x_i\land x_j=a$ for all $i\neq j$. We say that $S$ is lower-closed if $(x_i\in S,{y\in P},y\leq x_i)\Rightarrow y\in S$. We say that $S$ is meet closed if $x_i,x_j\in S\Rightarrow x_i\land x_j\in S$. It is clear that a lower-closed set is always meet closed but the converse does not hold.

\begin{definition}\label{def:2.2} We say that $f$ is a semimultiplicative function on $P$ if
\begin{equation}\label{eq:2.1}
f(x)f(y)=f(x\land y)f(x\lor y)
\end{equation}
for all $x,y\in P$.
\end{definition}

The concept of a semimultiplicative function on $P$ is a generalization of the concept of a semimultiplicative arithmetical function, see \cite[p.\ 49]{Rea} or \cite[p.~237]{Siva}. Let $f(x)\neq0$ for all $x\in P$. Then the function $\frac{1}{f}$ on $P$ is defined by $\left(\frac{1}{f}\right)(x)={1}/{f(x)}$. If $g$ is an incidence function of $P$, the incidence function $\frac{1}{g}$ of $P$ is defined similarly. One can easily show that $f$ is semimultiplicative if and only if $\frac{1}{f}$ is semimultiplicative. We associate each $f(z)$ with incidence function value $f(0, z)$. For example, by $(f*\mu)(z)$ we mean the convolution $$(f*\mu)(0,z)=\sum_{0\leq w\leq z}f(0,w)\mu(w,z).$$

\section{An inductive method}
\label{sect:3}
\setcounter{equation}{0}

In this section we provide an inductive method for constructing meet closed sets $S$ on which join matrices $[S]_f$ are nonsingular under certain conditions on $f$. The inductive method arises from the idea to construct meet closed sets element by element from the bottom up, see Definition \ref{def:3.1}.

{\sl Throughout the rest of this paper $(P,\leq)=(P,\land,\lor)$ is a lattice,  $S=\{x_1,x_2,\ldots,x_n\}$ is a meet closed subset of $P$ such that $x_i<x_j\Rightarrow i<j$ holds and $f$ is a semimultiplicative function on $P$ such that $f(x)\neq0$ for all $x\in P$.} 

Now, by using the semimultiplicativity of $f$, we may write 
\begin{equation}\label{eq:3.3}
[S]_f=\Delta_{S,f}(S)_{\frac{1}{f}}\Delta_{S,f},
\end{equation}
where $\Delta_{S,f}=\mathrm{diag}(f(x_1),f(x_2),\ldots,f(x_n))$ (see \cite[Theorem 6.1]{AltTugHau}, \cite[Theorem 6.1]{MatHau1} and \cite[Lemmas 5.1 and 5.2]{KorHau2}). Since $f(x_i)\neq0$ for all $i=1,2,\ldots,n$, the matrix $\Delta_{S,f}$ is clearly invertible. Therefore $[S]_f$ is invertible if and only if $(S)_{\frac{1}{f}}$ is invertible.

Let $S_i=\{x_1, x_2,\ldots, x_i\}$ for $i=1,2,\ldots,n$. Then $S_1\subset S_2\subset \cdots\subset S_n=S$ is a finite sequence of meet closed sets on $(P,\le)$ and lower-closed sets on $(S,\le)$. The values of the corresponding M\"obius function $\mu_S$ can be easily evaluated by using the recursion 
\begin{eqnarray}\label{eq:mu_Srec}
&&\mu_S(x_i, x_i)=1,\\
&&\mu_S(x_i, x_j)=-\sum_{k=i}^{j-1} \mu_S(x_i, x_k)=-\sum_{k=i+1}^{j} \mu_S(x_k, x_j),\ i<j, \nonumber
\end{eqnarray}
see \cite[p.~141]{Aig} or \cite[p.~116]{Stan}. Note that $\mu_S=\mu_{S_{i}}$ on $(S_{i},\le)$ and the convolutions on $(S_{i},\le)$ and  $(S,\le)$ are equal if the arguments belong to $S_{i}$. Thus for each $i\geq2$ we have
\begin{align}\label{eq:detformula}
\det(S_{i})_{\frac{1}{f}}&=\prod_{k=1}^{i}\left(\tfrac{1}{f}\ast_S \mu_S\right)(x_k)=\left(\tfrac{1}{f}\ast_S \mu_S\right)(x_{i})\prod_{k=1}^{i-1}\left(\tfrac{1}{f}\ast_S \mu_S\right)(x_k)\notag\\
&=\left(\tfrac{1}{f}\ast_S \mu_S\right)(x_{i})\det(S_{i-1})_{\frac{1}{f}}
\end{align}
(see \cite[Theorem 4.2]{AltTugHau} and \cite[Corollary 2]{Hau96}).

From \eqref{eq:3.3} and \eqref{eq:detformula} we see that if $[S_{i}]_f$ is invertible, then also $(S_{i})_{\frac{1}{f}}$, $(S_{i-1})_{\frac{1}{f}}$ and $[S_{i-1}]_f$ are invertible. Conversely, let $[S_{i-1}]_f$ be invertible. We below consider which elements of $P$, denoted as $x_{i}$, could be added to $S_{i-1}$ so that also $[S_{i}]_f$ is invertible.

\begin{definition}\label{def:3.1}
Let $S_0=\emptyset$ and $i\geq 1$. Consider the sets $S_{i-1}$ and $S_i=S_{i-1}\cup\{x_{i}\}$.
\begin{itemize}
\setlength{\itemindent}{1.1em}
\item[\quad\rm(M$_{m_i,i}$)] Let $m_i$ be the greatest integer such that $x_{i_1},x_{i_2},\ldots,x_{i_{m_i}}\in S_{i-1}$ are covered by $x_{i}$ in $S_{i}$.
\end{itemize}
If {\rm(M$_{m_i,i}$)} holds, then we say that $S_{i}$ is constructed from $S_{i-1}$ by the method {\rm(M$_{m_i,i}$)}. Further, if
\begin{itemize}
\setlength{\itemindent}{1em}
\item[\rm(C$_{m_i,i}$)] $({\textstyle\frac{1}{f}}*_S\mu_S)(x_{i})\ne 0$,
\end{itemize}
then we say that $S_{i}$ is constructed from $S_{i-1}$ by the method {\rm(M$_{m_i,i}$)} under the condition {\rm(C$_{m_i,i}$)}.
\end{definition}

\begin{remark}
We always must have $m_1=0$ and $m_2=m_3=1$. For example, the condition {\rm(C$_{0,1}$)} only states the triviality $\frac{1}{f}(x_1)\neq0$ whereas {\rm(C$_{1,2}$)} means that $\frac{1}{f}(x_2)-\frac{1}{f}(x_1)\neq0$.
\end{remark}

\begin{theorem}\label{th:3.1}
Let $i\geq2$ and $S_{i}$ be constructed from $S_{i-1}$ by {\rm(M$_{m_i,i}$)} under {\rm(C$_{m_i,i}$)}. Then $[S_{i}]_f$ is invertible if and only if $[S_{i-1}]_f$ is invertible.
\end{theorem}

\begin{proof}
Theorem \ref{th:3.1} is a direct consequence of \eqref{eq:3.3}, \eqref{eq:detformula} and Definition \ref{def:3.1}.
\end{proof}

The method (M$_{1,i}$) in Definition \ref{def:3.1} allows us to add an element $x_{i}$ above $x_{i_1}$ if $x_{i}$ covers $x_{i_1}$ in $S_{i}$. The method (M$_{2,i}$) allows us to join together two incomparable elements $x_{i_1},x_{i_2}$ with $x_{i}$ if $x_{i}$ covers both $x_{i_1}$ and $x_{i_2}$. The method (M$_{3,i}$) concerns three incomparable elements $x_{i_1},x_{i_2},x_{i_3}$ and so on. The condition (C$_{m_i,i}$) can be written as 

\begin{equation}\label{eq:Cm2}
\frac{1}{f(x_{i})}\ne -\sum_{k=1}^{i-1} {\textstyle\frac{1}{f}}(x_{k})\mu_S(x_k, x_{i}).
\end{equation}

For $m_i=1,2$ using the recursive properties of $\mu_S$, see $\eqref{eq:mu_Srec}$, we easily obtain 

\begin{itemize}
\setlength{\itemindent}{0.5em}
\item[\rm (C$_{1,i}$)] $\quad f(x_{i})\neq f(x_{i_1})$,
\item[\rm (C$_{2,i}$)] $\quad \frac{1}{f(x_{i})}\neq\frac{1}{f}(x_{i_1})+\frac{1}{f}(x_{i_2})-\frac{1}{f}(x_{i_1}\land x_{i_2})$.
\end{itemize}

By semimultiplicativity, (C$_{2,i}$) can be written without any meets as 

\begin{equation}\label{equ:3.7}
f(x_{i+1})\neq f(x_{i_1})f(x_{i_2})/[f(x_{i_1})+f(x_{i_2})-f(x_{i_1}\lor x_{i_2})]
\end{equation}
whenever the denominator is nonzero. Each meet closed set $S$ can be constructed inductively by a finite sequence (M$_{m_1,1}$), (M$_{m_2,2}$), $\ldots$ , (M$_{m_{n},n}$) (often there are multiple different ways to construct a given set $S$ but the sequence $(m_1,m_2,\ldots,m_n)$ is in fact unique up to ordering). Thus we have the following theorem.

\begin{theorem}\label{th:3.2}
Let $S$ be constructed inductively by a method sequence
$$(\mathrm{M}_{m_1,1}),(\mathrm{M}_{m_2,2}),\ldots,(\mathrm{M}_{m_{n},n}).$$
Then $[S]_f$ is invertible if and only if the condition sequence
$$(\mathrm{C}_{m_1,1}),(\mathrm{C}_{m_2,2}),\ldots,(\mathrm{C}_{m_{n},n})$$
holds.
\end{theorem}

\section{Classification of functions on the basis of the used methods}\label{sect:4}\setcounter{equation}{0}

Let ${\cal F}$ denote the class of all semimultiplicative functions $f$ on $P$ such that $f(x)\neq0$ for all $x\in P$. We divide ${\cal F}$ into subclasses on the basis of the numbers $m_i=1, 2,\ldots$ in the method sequences $({\rm M}_{m_1,1}), ({\rm M}_{m_2,2}),\ldots,$ $({\rm M}_{m_{n},n})$. We introduce two kinds of subclasses ${\cal F}_k$ and ${\cal G}_{k, n}$. The classes ${\cal F}_k$ are smaller than the classes ${\cal G}_{k, n}$ and are introduced to get the presentation shorter. Let ${\cal S}_{k,n}$ denote the class of all meet closed subsets $S$ of $P$ possessing the structure as described in Figure 1. The white points in Figure 1 stand for the last added elements $x_n$. Note that although $x_k$ would be the supremum of $x_i$ and $x_j$ in $(S,\preceq)$, it does not necessarily represent the element $x_i\vee x_j\in P$. In the notation ${\cal S}_{k,n}$, the number $k$ comes from the last used method $({\rm M}_{k,n})$ in constructing the set $S\in {\cal S}_{k,n}$  (that is, the last added element $x_n$ covers $k$ but no more incomparable elements $x_{i_1}, x_{i_2},\ldots, x_{i_k}$ in $S$), and the letter $n$ just indicates the number of elements in $S\in {\cal S}_{k,n}$. For the pair $k=4, n=7$ we should distinguish two distinct classes ${\cal S}_{4, 7}^{(1)}$ and ${\cal S}_{4, 7}^{(2)}$. We are now in a position to define the function classes ${\cal G}_{k, n}$.

\begin{definition}
For each  ${\cal S}_{k,n}$ in Figure 1 let
$${\cal G}_{k, n}=\{f\in{\cal F} \mid \forall S\in {\cal S}_{k,n}:(\textstyle\frac{1}{f}*_S\mu_S)(x_n)\ne 0\}.$$
In addition,
$${\cal G}_{4, 7}^{(j)}=\{f\in{\cal F} \mid \forall S\in {\cal S}_{4, 7}^{(j)}: (\textstyle\frac{1}{f}*_S\mu_S)(x_7)\ne 0\},\ j=1, 2. $$
\end{definition}

The condition $(\textstyle\frac{1}{f}*\mu_S)(x_n)\ne 0$ means that the last condition $({\rm C}_{m_{n},n})$ in the condition sequence in Theorem \ref{th:3.2} holds.

\setcounter{subfigure}{0}
\begin{figure}[htb!]
\centering
\subfigure[${\cal S}_{1,2}$]
{{\scalefont{0.4}
\begin{tikzpicture}[scale=0.55]
\draw (0,1)--(0,1.9);
\draw [fill] (0,1) circle [radius=0.1];
\draw (0,2) circle [radius=0.1];
\node [right] at (0,1) {\textrm{-}$1$};
\node [right] at (0,2) {$1$};
\node [right] at (1,2) {\ };
\node [right] at (-1,2) {\ };
\end{tikzpicture}}
}
\subfigure[${\cal S}_{2,4}$]
{{\scalefont{0.4}
\begin{tikzpicture}[scale=0.55]
\draw (1,0)--(0.25,1)--(0.92,1.92);
\draw (1,0)--(1.75,1)--(1.08,1.92);
\draw [fill] (0.25,1) circle [radius=0.1];
\draw [fill] (1.75,1) circle [radius=0.1];
\draw [fill] (1,0) circle [radius=0.1];
\draw (1,2) circle [radius=0.1];
\node [right] at (1,0) {$1$};
\node [right] at (0.25,1) {\textrm{-}$1$};
\node [right] at (1.75,1) {\textrm{-}$1$};
\node [right] at (1,2) {$1$};
\end{tikzpicture}}
}
\subfigure[${\cal S}_{3,5}$]
{{\scalefont{0.4}
\begin{tikzpicture}[scale=0.55]
\draw (1,0)--(0,1)--(0.92,1.92);
\draw (1,0)--(1,1.9);
\draw (1,0)--(2,1)--(1.08,1.92);
\draw [fill] (0,1) circle [radius=0.1];
\draw [fill] (1,1) circle [radius=0.1];
\draw [fill] (2,1) circle [radius=0.1];
\draw [fill] (1,0) circle [radius=0.1];
\draw (1,2) circle [radius=0.1];
\node [right] at (1,0) {$2$};
\node [right] at (0,1) {\textrm{-}$1$};
\node [right] at (1,1) {\textrm{-}$1$};
\node [right] at (2,1) {\textrm{-}$1$};
\node [right] at (1,2) {$1$};
\end{tikzpicture}}
}
\subfigure[${\cal S}_{3,6}$]
{{\scalefont{0.4}
\begin{tikzpicture}[scale=0.55]
\draw (2,0)--(1.3,0.7)--(0.9,1.4)--(1.65,2.4);
\draw (1.3,0.7)--(1.7,1.4)--(1.7,2.4);
\draw (2,0)--(2.8,1.4)--(1.75,2.4);
\draw [fill] (2,0) circle [radius=0.1];
\draw [fill] (1.3,0.7) circle [radius=0.1];
\draw [fill] (0.9,1.4) circle [radius=0.1];
\draw [fill] (2.8,1.4) circle [radius=0.1];
\draw [fill] (1.7,1.4) circle [radius=0.1];
\draw (1.7,2.5) circle [radius=0.1];
\node [right] at (2,0) {$1$};
\node [right] at (1.3,0.7) {$1$};
\node [right] at (0.9,1.4) {\textrm{-}$1$};
\node [right] at (1.7,1.4) {\textrm{-}$1$};
\node [right] at (2.8,1.4) {\textrm{-}$1$};
\node [right] at (1.7,2.5) {$1$};
\end{tikzpicture}}
}
\subfigure[${\cal S}_{3,7}$]
{{\scalefont{0.4}
\begin{tikzpicture}[scale=0.55]
\draw (1,0)--(0,1)--(0,2)--(0.92,2.92);
\draw (0,1)--(1,2)--(1,2.9);
\draw (1,0)--(2,1)--(2,2)--(1.08,2.92);
\draw (2,1)--(1,2);
\draw [fill] (0,1) circle [radius=0.1];
\draw [fill] (2,1) circle [radius=0.1];
\draw [fill] (1,0) circle [radius=0.1];
\draw [fill] (0,2) circle [radius=0.1];
\draw [fill] (1,2) circle [radius=0.1];
\draw [fill] (2,2) circle [radius=0.1];
\draw (1,3) circle [radius=0.1];
\node [right] at (1,0) {$0$};
\node [right] at (0,1) {$1$};
\node [right] at (2,1) {$1$};
\node [right] at (0,2) {\textrm{-}$1$};
\node [right] at (1,2) {\textrm{-}$1$};
\node [right] at (2,2) {\textrm{-}$1$};
\node [right] at (1,3) {$1$};
\end{tikzpicture}}
}
\subfigure[${\cal S}_{3,8}$]
{{\scalefont{0.4}
\begin{tikzpicture}[scale=0.55]
\draw (1,0)--(0,1)--(0,2)--(0.92,2.92);
\draw (1,0)--(1,1)--(0,2);
\draw (0,1)--(1,2)--(1,2.9);
\draw (1,0)--(2,1)--(2,2)--(1.08,2.92);
\draw (1,1)--(2,2);
\draw (2,1)--(1,2);
\draw [fill] (0,1) circle [radius=0.1];
\draw [fill] (1,1) circle [radius=0.1];
\draw [fill] (2,1) circle [radius=0.1];
\draw [fill] (1,0) circle [radius=0.1];
\draw [fill] (0,2) circle [radius=0.1];
\draw [fill] (1,2) circle [radius=0.1];
\draw [fill] (2,2) circle [radius=0.1];
\draw (1,3) circle [radius=0.1];
\node [right] at (1,0) {\textrm{-}$1$};
\node [right] at (0,1) {$1$};
\node [right] at (1,1) {$1$};
\node [right] at (2,1) {$1$};
\node [right] at (0,2) {\textrm{-}$1$};
\node [right] at (1,2) {\textrm{-}$1$};
\node [right] at (2,2) {\textrm{-}$1$};
\node [right] at (1,3) {$1$};
\end{tikzpicture}}
}\subfigure[${\cal S}_{4,6}$]
{{\scalefont{0.4}
\begin{tikzpicture}[scale=0.55]
\draw (1,0)--(-0.25,1)--(0.92,1.92);
\draw (1,0)--(0.5,1)--(0.95,1.9);
\draw (1,0)--(1.5,1)--(1.05,1.9);
\draw (1,0)--(2.25,1)--(1.08,1.92);
\draw [fill] (-0.25,1) circle [radius=0.1];
\draw [fill] (0.5,1) circle [radius=0.1];
\draw [fill] (1.5,1) circle [radius=0.1];
\draw [fill] (1,0) circle [radius=0.1];
\draw [fill] (2.25,1) circle [radius=0.1];
\draw (1,2) circle [radius=0.1];
\node [right] at (1,0) {$3$};
\node [right] at (-0.25,1) {\textrm{-}$1$};
\node [right] at (1.5,1) {\textrm{-}$1$};
\node [right] at (0.5,1) {\textrm{-}$1$};
\node [right] at (2.25,1) {\textrm{-}$1$};
\node [right] at (1,2) {$1$};
\end{tikzpicture}}
}
\subfigure[${\cal S}_{4,7}^{(1)}$]
{{\scalefont{0.4}
\begin{tikzpicture}[scale=0.55]
\draw (2,0.5)--(1,1)--(0,2)--(1.42,2.92);
\draw (1,1)--(1,2)--(1.45,2.9);
\draw (2,0.5)--(3,2)--(1.58,2.92);
\draw (1,1)--(2,2)--(1.55,2.9);
\draw [fill] (2,0.5) circle [radius=0.1];
\draw [fill] (1,1) circle [radius=0.1];
\draw [fill] (0,2) circle [radius=0.1];
\draw [fill] (1,2) circle [radius=0.1];
\draw [fill] (2,2) circle [radius=0.1];
\draw [fill] (3,2) circle [radius=0.1];
\draw (1.5,3) circle [radius=0.1];
\node [right] at (1.5,3) {$1$};
\node [right] at (0,2) {\textrm{-}$1$};
\node [right] at (1,2) {\textrm{-}$1$};
\node [right] at (2,2) {\textrm{-}$1$};
\node [right] at (3,2) {\textrm{-}$1$};
\node [right] at (1.05,1.05) {$2$};
\node [right] at (2,0.5) {$1$};
\end{tikzpicture}}
}
\subfigure[${\cal S}_{4,7}^{(2)}$]
{{\scalefont{0.4}
\begin{tikzpicture}[scale=0.55]
\draw (2,0.5)--(1,1)--(0,2)--(1.42,2.92);
\draw (1,1)--(1,2)--(1.45,2.9);
\draw (2,0.5)--(3,2)--(1.58,2.92);
\draw (2,0.5)--(2,2)--(1.55,2.9);
\draw [fill] (2,0.5) circle [radius=0.1];
\draw [fill] (1,1) circle [radius=0.1];
\draw [fill] (0,2) circle [radius=0.1];
\draw [fill] (1,2) circle [radius=0.1];
\draw [fill] (2,2) circle [radius=0.1];
\draw [fill] (3,2) circle [radius=0.1];
\draw (1.5,3) circle [radius=0.1];
\node [right] at (1.5,3) {$1$};
\node [right] at (0,2) {\textrm{-}$1$};
\node [right] at (1,2) {\textrm{-}$1$};
\node [right] at (2,2) {\textrm{-}$1$};
\node [right] at (3,2) {\textrm{-}$1$};
\node [right] at (1.05,1.05) {$1$};
\node [right] at (2,0.5) {$2$};
\end{tikzpicture}}
}
\subfigure[${\cal S}_{5,7}$]
{{\scalefont{0.4}
\begin{tikzpicture}[scale=0.55]
\draw (1,0)--(-0.75,1)--(0.92,1.92);
\draw (1,0)--(0.2,1)--(0.92,1.92);
\draw (1,0)--(1,1.9);
\draw (1,0)--(1.8,1)--(1.08,1.92);
\draw (1,0)--(2.75,1)--(1.08,1.92);
\draw [fill] (0.2,1) circle [radius=0.1];
\draw [fill] (1,1) circle [radius=0.1];
\draw [fill] (1.8,1) circle [radius=0.1];
\draw [fill] (1,0) circle [radius=0.1];
\draw [fill] (2.75,1) circle [radius=0.1];
\draw [fill] (-0.75,1) circle [radius=0.1];
\draw (1,2) circle [radius=0.1];
\node [right] at (1,-0.1) {$4$};
\node [right] at (0.2,1) {\textrm{-}$1$};
\node [right] at (1,1) {\textrm{-}$1$};
\node [right] at (-0.75,1) {\textrm{-}$1$};
\node [right] at (2.75,1) {\textrm{-}$1$};
\node [right] at (1.8,1) {\textrm{-}$1$};
\node [right] at (1,2.05) {$1$};
\end{tikzpicture}}
}
\subfigure[${\cal S}_{n-2,n}$]
{{\scalefont{0.4}
\begin{tikzpicture}[scale=0.55]
\draw (1,0)--(-1.8,1)--(0.92,1.92);
\draw (1,0)--(-0.75,1)--(0.92,1.92);
\draw (1,0)--(0.2,1)--(0.92,1.92);
\draw (1,0)--(1.8,1)--(1.08,1.92);
\draw (1,0)--(2.75,1)--(1.08,1.92);
\draw (1,0)--(3.8,1)--(1.08,1.92);
\draw [fill] (-1.8,1) circle [radius=0.1];
\draw [fill] (3.8,1) circle [radius=0.1];
\draw [fill] (0.2,1) circle [radius=0.1];
\draw [fill] (1.8,1) circle [radius=0.1];
\draw [fill] (1,0) circle [radius=0.1];
\draw [fill] (2.75,1) circle [radius=0.1];
\draw [fill] (-0.75,1) circle [radius=0.1];
\draw (1,2) circle [radius=0.1];
\node [right] at (1,-0.1) {$n-3$};
\node [right] at (0.2,1) {\textrm{-}$1$};
\node at (1,1) {\ \ $\ldots$};
\node [right] at (-0.7,1) {\textrm{-}$1$};
\node [right] at (2.75,1) {\textrm{-}$1$};
\node [right] at (1.8,1) {\textrm{-}$1$};
\node [right] at (1,2.05) {$1$};
\node [right] at (-1.8,1) {\ \textrm{-}$1$};
\node [right] at (3.8,1) {\textrm{-}$1$};
\end{tikzpicture}}
}
\centerline{\bf Figure 1.}
\end{figure}

For each class ${\cal S}_{k,n}\ni S$ we have marked in Figure 1 the value of $\mu_S(x_i, x_n)$ next to each element $x_i$. The value of $\mu_S(x_i, x_n)$ can be easily seen by $(\ref{eq:mu_Srec})$. 

\begin{definition}\label{de:F_k}
For each $k=1, 2,\ldots$ let ${\cal F}_k$ denote the set of functions $f\in {\cal F}$ satisfying the condition sequence $({\rm C}_{m_1,1}),({\rm C}_{m_2,2}),\ldots,({\rm C}_{m_{n},n})$ for all meet closed subsets $S$ of $P$ such that $S$ can be constructed by $({\rm M}_{m_1,1}),({\rm M}_{m_2,2}),\ldots,$ $({\rm M}_{m_{n},n})$, where $m_1, m_2,\ldots, m_{n}\le k$.
\end{definition}

It is easy to see that ${\cal F}\supseteq{\cal F}_1\supseteq{\cal F}_2\supseteq{\cal F}_3\supseteq\cdots$ and more precisely
\begin{eqnarray}\label{eq:4.4}
&&{\cal F}_1={\cal G}_{1, 2}   =\{ f \mid \forall y,z\in P: y<z\Rightarrow f(y)\neq f(z)\}, \\
&&{\cal F}_2={\cal F}_1\cap {\cal G}_{2, 4}={\cal F}_1\cap \{ f \mid \forall\ {\rm antichains}\ y_1 ,y_2\in P,\forall z\in P:\nonumber \\
&&\quad\qquad   y_1\lor y_2\leq z\Rightarrow \textstyle\frac{1}{f}(z)\neq \frac{1}{f}(y_1)+\frac{1}{f}(y_2)-\frac{1}{f}(y_1\land y_2)\},\\
&&{\cal F}_3={\cal F}_2\cap {\cal G}_{3, 5}\cap{\cal G}_{3, 6}\cap{\cal G}_{3, 7}\cap {\cal G}_{3, 8}.
\end{eqnarray}

When adding the last element $x_n$ to the set $S_{n-1}$ the invertibility of $[S_{n}]_f$ depends only on the invertibility of $[S_{n-1}]_f$ and on the values $f(x_i)$ of $x_i$ such that $\mu_S(x_i,x_n)\neq0$. Thus when considering whether the condition ($C_{m_{n},n}$) is satisfied or not we can omit all elements $x_i$ with $\mu_S(x_i,x_n)=0$. It will turn out that when $n\leq 7$ we can omit most of the cases and restrict ourselves to the structures presented in Figure 1.

\begin{remark}
All the structures of $S$ mentioned here need not appear in a fixed lattice $(P,\leq)$, and thus the structure of $(P,\leq)$ also has a bearing on the possibility of the invertibility.
\end{remark}

\section{Chains, \texorpdfstring{$x_1$}{x1}-sets and a related class}\label{sect:5.chains}\setcounter{equation}{0}

In this section we consider invertibility of $[S]_f$ on certain sets $S$ which we use frequently in the lattice-theoretic generalization of the Bourque-Ligh conjecture in Section 6.

\begin{theorem}\label{th:3.3}
If $S$ is a chain, then $[S]_f$ is invertible if and only if $f(x_k)\neq f(x_{k-1})$ for $k=2,3,\ldots,n$. If $S$ is an $x_1$-set, then $[S]_f$ is invertible if and only if $f(x_k)\neq f(x_1)$ for $k=2,3,\ldots, n$.
\end{theorem}

\begin{proof}
Chains and $x_1$-sets are constructed using the methods $({\rm M}_{1,i})$ only. By Theorem \ref{th:3.2}, we obtain Theorem \ref{th:3.3} taking the appropriate conditions $({\rm C}_{1,i})$.
\end{proof}

\begin{remark}
It is easy to see that if the set $S$ is meet closed and can be constructed by using only the methods $({\rm M}_{1,i})$, then $f\in{\cal F}_1$ is a sufficient condition for the invertibility of $(S)_{\frac{1}{f}}$ and $(S)_f$ and, provided that $f$ is semimultiplicative with nonzero values, also for the invertibility of $[S]_f$ and $[S]_{\frac{1}{f}}$. In this case the Hasse diagram of the set $S$ considered as an undirected graph is a tree, and the positive definiteness of the matrix $(S)_f$ has an interesting connection to the properties of the function $f$, see \cite[Theorems 4.1 and 4.2]{MatHau2}. 
\end{remark}

\begin{corollary}\label{cor:3.1}
Let $(P,\leq)=({\mathbb Z}_+,\mid)$. If $S$ is a (divisor) chain or an $x_1$-set, then $[S]$ is invertible.
\end{corollary}

\begin{proof}
The arithmetical function $N$ fulfills the conditions in Theorem \ref{th:3.3}.
\end{proof}

Note that the conditions of Theorem \ref{th:3.3} also imply the invertibility of the associated meet matrix $(S)_f$, see \cite[Corollary~2]{Hau96}. The requirement of semimultiplicativity of $f$ in the first part of Theorem \ref{th:3.3} is irrelevant, since any $f$ is semimultiplicative on chains.

One important class of meet closed sets (termed as ${\cal S}_{n-2, n}$, see Figure 1) is constructed by adding an upper bound to an $x_1$-set.

\begin{theorem}\label{th:3.4}
Let $n\geq 3$. Let $S\in {\cal S}_{n-2, n}$, i.e., $S_{n-1}$ is an $x_1$-set and $x_1\lor\cdots\lor x_{n-1}\leq x_n$. Then $[S]_f$ is invertible if and only if $f(x_k)\neq f(x_1)$ for $k=2,3,\ldots, n-1$ and $$\frac{1}{f(x_n)}\neq\left(\sum_{k=2}^{n-1}\frac{1}{f(x_k)}\right)-\frac{n-3}{f(x_1)}.$$
\end{theorem}

\begin{proof}
Since $S$ can be constructed from an $x_1$-set $S_{n-1}$ by (M$_{n-2,n}$), then the conditions are those mentioned in Theorem \ref{th:3.3} for $S_{n-1}$ together with condition (C$_{n-2,n}$). Using $(\ref{eq:Cm2})$ and the values $\mu_S(x_k, x_n)$ of ${\cal S}_{n-2, n}$ in Figure 1 we obtain 
\begin{equation}
\frac{1}{f(x_n)}\neq\frac{1}{f(x_{n-1})}+\cdots+ \frac{1}{f(x_2)}-\frac{n-3}{f(x_1)}.\nonumber
\end{equation}
\end{proof}

\begin{corollary}\label{cor:3.2}
Let $(P,\leq)=({\mathbb Z}_+,\mid)$ and let $n\geq3$. If $S_{n-1}$ is an $x_1$-set and $\hbox{\rm lcm}(S_{n-1})\mid x_n$, i.e.,~$S\in {\cal S}_{n-2, n}$, then $[S]$ is invertible.
\end{corollary}

\begin{proof}
It suffices to prove that $N\in {\cal G}_{n-2, n}$. The case $n=3$ follows from Corollary \ref{cor:3.1}, so we may assume that $n\geq4$. Now $x_1=\hbox{gcd}(x_i,x_j)$ for all $2\leq i<j\leq n-1$. Thus for $i=2,3,\ldots,n-1$ we have $x_i=a_ix_1$, where $a_i$'s are distinct and $a_i\geq2$ for each $i$. Thus we have
\begin{equation*}\frac{1}{x_n}+\frac{n-3}{x_1}-\sum_{k=2}^{n-1}\frac{1}{x_k} =\frac{1}{x_n}+\frac{1}{x_1}\left((n-3)-\sum_{k=2}^{n-1}\frac{1}{a_k}\right)>0,
\end{equation*}
since
\[
\sum_{k=2}^{n-1}\frac{1}{a_k}<\sum_{k=2}^{n-1}\frac{1}{2}=\frac{n-2}{2}\leq n-3.
\]
Thus $N\in {\cal G}_{n-2, n}$.
\end{proof}

\begin{remark}\label{rem:B-L} Let $(P,\le)=({\mathbb Z}_+, \mid)$. Since $N\in {\cal F}_2$, we see that if $S$ is any gcd-closed set constructed by $({\rm M}_{1,i})$ and $({\rm M}_{2,i})$ repeatedly, then the LCM matrix $[S]$ is invertible, see Corollaries \ref{cor:3.1} and \ref{cor:3.2} In particular, by Corollary \ref{cor:3.2} we also have $N\in {\cal G}_{2,4}$, $N\in {\cal G}_{3,5}$, $N\in {\cal G}_{4,6}$ and $N\in {\cal G}_{5,7}$.
\end{remark}

\section{The Bourque-Ligh conjecture}\label{sect:6}

Bourque and Ligh \cite{Bour92} conjectured that the LCM matrix $[S]$ is invertible on any gcd-closed set $S$. It is known that this conjecture holds for $n\le 7$ and does not generally hold for $n\ge 8$. A number-theoretic proof of this solution has been given in \cite{Hong99}. We here provide a lattice-theoretic proof. We go through all meet closed sets $S$ (up to isomorphism) with $n=1, 2,\ldots, 7$ elements, and applying the conditions (C$_{m_i,i}$) we study the invertibility of the join matrix $[S]_f$ on $S$ in any lattice. When we take $(P,\le)=({\mathbb Z}_+, \mid)$ and $f=N$ we obtain the solution of the Bourque-Ligh conjecture given in \cite{Hong99}. In principle this is a simple method, since at least for small $n$ the sets $S$ are easy to classify on the basis of their incomparable elements and the conditions (C$_{m_i,i}$) are easy to evaluate applying $(\ref{eq:Cm2})$, the Hasse diagram of $S$ and the recursive properties of $\mu_S$. It would be easy to derive necessary and sufficient conditions for the invertibility of the join matrix $[S]_f$ on $S$ in any lattice, but for the sake of brevity in we present only sufficient conditions.

\subsection{Cases \texorpdfstring{$n=1, 2, 3, 4, 5$}{n=1,2,3,4,5}}\label{subsect:4.3}

We begin by constructing recursively all possible meet closed sets with at most $5$ elements, see Figure 2. If all meet semilattices with $n$ elements are known, then a simple but laborous way to obtain all possible meet semilattices with $n+1$ elements is first to determine all possible ways to add a maximal element to them and then to eliminate repetitions. The semilattices are then classified based on the largest $m_i$ in the methods (M$_{m_i,i}$) used to construct each semilattice. Most of them are constructed by using (M$_{1,i}$) only, but for some of them also (M$_{2,i}$) or even (M$_{3,i}$) is needed.

\setcounter{figure}{0}
\setcounter{subfigure}{0}
\begin{figure}[htb!]
\centering
\subfigure[$1_{\rm A}$]
{{\scalefont{0.4}
\begin{tikzpicture}[scale=0.55]
\draw (0,2) circle [radius=0.1];
\node [right] at (0,2) {$1$};
\node [right] at (1,2) {\ };
\node [right] at (-1,2) {\ };
\end{tikzpicture}}
}
\subfigure[$2_{\rm A}$]
{{\scalefont{0.4}
\begin{tikzpicture}[scale=0.55]
\draw (0,1)--(0,1.9);
\draw [fill] (0,1) circle [radius=0.1];
\draw (0,2) circle [radius=0.1];
\node [right] at (0,1) {\textrm{-}$1$};
\node [right] at (0,2) {$1$};
\node [right] at (1,2) {\ };
\node [right] at (-1,2) {\ };
\end{tikzpicture}}
}
\subfigure[$3_{\rm A}$]
{{\scalefont{0.4}
\begin{tikzpicture}[scale=0.55]
\draw (0,0)--(0,1.9);
\draw [fill] (0,1) circle [radius=0.1];
\draw [fill] (0,0) circle [radius=0.1];
\draw (0,2) circle [radius=0.1];
\node [right] at (0,1) {\textrm{-}$1$};
\node [right] at (0,2) {$1$};
\node [right] at (0,0) {$0$};
\node [right] at (1,2) {\ };
\node [right] at (-1,2) {\ };
\end{tikzpicture}}
}
\subfigure[$3_{\rm B}$]
{{\scalefont{0.4}
\begin{tikzpicture}[scale=0.55]
\draw (0,2)--(-0.5,3);
\draw (0,2)--(0.45,2.9);
\draw (0.5,3) circle [radius=0.1];
\draw [fill] (0,2) circle [radius=0.1];
\draw [fill] (-0.5,3) circle [radius=0.1];
\node [right] at (0,2) {\textrm{-}$1$};
\node [right] at (-0.5,3) {$0$};
\node [right] at (0.5,3) {$1$};
\end{tikzpicture}}
}
\subfigure[$4_{\rm A}$]
{{\scalefont{0.4}
\begin{tikzpicture}[scale=0.55]
\draw (0,1)--(0,3.9);
\draw [fill] (0,1) circle [radius=0.1];
\draw [fill] (0,2) circle [radius=0.1];
\draw [fill] (0,3) circle [radius=0.1];
\draw (0,4) circle [radius=0.1];
\node [right] at (0,1) {$0$};
\node [right] at (0,2) {$0$};
\node [right] at (0,3) {\textrm{-}$1$};
\node [right] at (0,4) {$1$};
\node [right] at (1,2) {\ };
\node [right] at (-1,2) {\ };
\end{tikzpicture}}
}
\subfigure[$4_{\rm B}$]
{{\scalefont{0.4}
\begin{tikzpicture}[scale=0.55]
\draw (0,1)--(0,2)--(-0.5,3);
\draw (0,2)--(0.45,2.9);
\draw [fill] (0,1) circle [radius=0.1];
\draw [fill] (0,2) circle [radius=0.1];
\draw [fill] (-0.5,3) circle [radius=0.1];
\draw (0.5,3) circle [radius=0.1];
\node [right] at (0,1) {$0$};
\node [right] at (0,2) {\textrm{-}$1$};
\node [right] at (-0.5,3) {$0$};
\node [right] at (0.5,3) {$1$};
\end{tikzpicture}}
}
\subfigure[$4_{\rm C}$]
{{\scalefont{0.4}
\begin{tikzpicture}[scale=0.55]
\draw (0,2)--(-0.5,3)--(-0.5,3.9);
\draw (0,2)--(0.45,2.9);
\draw (0.5,3) circle [radius=0.1];
\draw [fill] (0,2) circle [radius=0.1];
\draw [fill] (-0.5,3) circle [radius=0.1];
\draw [fill] (-0.5,4) circle [radius=0.1];
\node [right] at (-0.5,4) {$0$};
\node [right] at (0,2) {\textrm{-}$1$};
\node [right] at (-0.5,3) {$0$};
\node [right] at (0.5,3) {$1$};
\end{tikzpicture}}
}
\subfigure[$4_{\rm D}$]
{{\scalefont{0.4}
\begin{tikzpicture}[scale=0.55]
\draw (1,0)--(0,1);
\draw (1,0)--(1,0.9);
\draw (1,0)--(2,1);
\draw [fill] (0,1) circle [radius=0.1];
\draw [fill] (2,1) circle [radius=0.1];
\draw [fill] (1,0) circle [radius=0.1];
\draw (1,1) circle [radius=0.1];
\node [right] at (1,0) {\textrm{-}$1$};
\node [right] at (0,1) {$0$};
\node [right] at (1,1) {$1$};
\node [right] at (2,1) {$0$};
\end{tikzpicture}}
}
\subfigure[$4_{\rm E}$]
{{\scalefont{0.4}
\begin{tikzpicture}[scale=0.55]
\draw (1,0)--(0.25,1)--(0.92,1.92);
\draw (1,0)--(1.75,1)--(1.08,1.92);
\draw [fill] (0.25,1) circle [radius=0.1];
\draw [fill] (1.75,1) circle [radius=0.1];
\draw [fill] (1,0) circle [radius=0.1];
\draw (1,2) circle [radius=0.1];
\node [right] at (1,0) {$1$};
\node [right] at (0.25,1) {\textrm{-}$1$};
\node [right] at (1.75,1) {\textrm{-}$1$};
\node [right] at (1,2) {$1$};
\end{tikzpicture}}
}
\subfigure[$5_{\rm A}$]
{{\scalefont{0.4}
\begin{tikzpicture}[scale=0.55]
\draw (0,1)--(0,4.9);
\draw [fill] (0,1) circle [radius=0.1];
\draw [fill] (0,2) circle [radius=0.1];
\draw [fill] (0,3) circle [radius=0.1];
\draw [fill] (0,4) circle [radius=0.1];
\draw (0,5) circle [radius=0.1];
\node [right] at (0,1) {$0$};
\node [right] at (0,2) {$0$};
\node [right] at (0,4) {\textrm{-}$1$};
\node [right] at (0,3) {$0$};
\node [right] at (0,5) {$1$};
\node [right] at (1,2) {\ };
\node [right] at (-1,2) {\ };
\end{tikzpicture}}
}
\subfigure[$5_{\rm B}$]
{{\scalefont{0.4}
\begin{tikzpicture}[scale=0.55]
\draw (0,1)--(0,3)--(-0.5,4);
\draw (0,3)--(0.45,3.9);
\draw [fill] (0,1) circle [radius=0.1];
\draw [fill] (0,2) circle [radius=0.1];
\draw [fill] (0,3) circle [radius=0.1];
\draw [fill] (-0.5,4) circle [radius=0.1];
\draw (0.5,4) circle [radius=0.1];
\node [right] at (0,1) {$0$};
\node [right] at (0,2) {$0$};
\node [right] at (-0.5,4) {$0$};
\node [right] at (0,3) {\textrm{-}$1$};
\node [right] at (0.5,4) {$1$};
\end{tikzpicture}}
}
\subfigure[$5_{\rm C}$]
{{\scalefont{0.4}
\begin{tikzpicture}[scale=0.55]
\draw (0,1)--(0,2)--(-0.5,3)--(-0.5,4);
\draw (0,2)--(0.45,2.9);
\draw [fill] (0,1) circle [radius=0.1];
\draw [fill] (0,2) circle [radius=0.1];
\draw [fill] (-0.5,4) circle [radius=0.1];
\draw [fill] (-0.5,3) circle [radius=0.1];
\draw (0.5,3) circle [radius=0.1];
\node [right] at (0,1) {$0$};
\node [right] at (0,2) {\textrm{-}$1$};
\node [right] at (-0.5,3) {$0$};
\node [right] at (0.5,3) {$1$};
\node [right] at (-0.5,4) {$0$};
\end{tikzpicture}}
}
\subfigure[$5_{\rm D}$]
{{\scalefont{0.4}
\begin{tikzpicture}[scale=0.55]
\draw (0,2)--(-0.5,3)--(-0.5,4)--(-0.5,4.9);
\draw (0,2)--(0.45,2.9);
\draw (0.5,3) circle [radius=0.1];
\draw [fill] (0,2) circle [radius=0.1];
\draw [fill] (-0.5,3) circle [radius=0.1];
\draw [fill] (-0.5,4) circle [radius=0.1];
\draw [fill] (-0.5,5) circle [radius=0.1];
\node [right] at (-0.5,4) {$0$};
\node [right] at (0,2) {\textrm{-}$1$};
\node [right] at (-0.5,3) {$0$};
\node [right] at (-0.5,5) {$0$};
\node [right] at (0.5,3) {$1$};
\node [right] at (1,2) {\ };
\node [right] at (-1,2) {\ };
\end{tikzpicture}}
}
\subfigure[$5_{\rm E}$]
{{\scalefont{0.4}
\begin{tikzpicture}[scale=0.55]
\draw (1,-1)--(1,0)--(0,1);
\draw (1,0)--(1,0.9);
\draw (1,0)--(2,1);
\draw [fill] (0,1) circle [radius=0.1];
\draw [fill] (1,-1) circle [radius=0.1];
\draw [fill] (2,1) circle [radius=0.1];
\draw [fill] (1,0) circle [radius=0.1];
\draw (1,1) circle [radius=0.1];
\node [right] at (1,0) {\textrm{-}$1$};
\node [right] at (0,1) {$0$};
\node [right] at (1,1) {$1$};
\node [right] at (2,1) {$0$};
\node [right] at (1,-1) {$0$};
\end{tikzpicture}}
}
\subfigure[$5_{\rm F}$]
{{\scalefont{0.4}
\begin{tikzpicture}[scale=0.55]
\draw (0.5,1)--(0,2)--(-0.5,3);
\draw (0,2)--(0.45,2.9);
\draw (0.5,1)--(0.95,1.9);
\draw (1,2) circle [radius=0.1];
\draw [fill] (0.5,1) circle [radius=0.1];
\draw [fill] (0,2) circle [radius=0.1];
\draw [fill] (-0.5,3) circle [radius=0.1];
\draw [fill] (0.5,3) circle [radius=0.1];
\node [right] at (0.5,1) {\textrm{-}$1$};
\node [right] at (0,2) {$0$};
\node [right] at (-0.5,3) {$0$};
\node [right] at (0.5,3) {$0$};
\node [right] at (1,2) {$1$};
\end{tikzpicture}}
}
\subfigure[$5_{\rm G}$]
{{\scalefont{0.4}
\begin{tikzpicture}[scale=0.55]
\draw (0,2)--(-0.5,3)--(-0.5,3.9);
\draw (0,2)--(0.5,3)--(0.5,3.9);
\draw (0.5,4) circle [radius=0.1];
\draw [fill] (0,2) circle [radius=0.1];
\draw [fill] (-0.5,3) circle [radius=0.1];
\draw [fill] (0.5,3) circle [radius=0.1];
\draw [fill] (-0.5,4) circle [radius=0.1];
\node [right] at (-0.5,4) {$0$};
\node [right] at (0,2) {$0$};
\node [right] at (-0.5,3) {$0$};
\node [right] at (0.5,3) {\textrm{-}$1$};
\node [right] at (0.5,4) {$1$};
\end{tikzpicture}}
}
\subfigure[$5_{\rm H}$]
{{\scalefont{0.4}
\begin{tikzpicture}[scale=0.55]
\draw (1,0)--(0,1);
\draw (1,0)--(1,1.9);
\draw (1,0)--(2,1);
\draw [fill] (0,1) circle [radius=0.1];
\draw [fill] (2,1) circle [radius=0.1];
\draw [fill] (1,0) circle [radius=0.1];
\draw [fill] (1,1) circle [radius=0.1];
\draw (1,2) circle [radius=0.1];
\node [right] at (1,0) {$0$};
\node [right] at (0,1) {$0$};
\node [right] at (1,1) {\textrm{-}$1$};
\node [right] at (1,2) {$1$};
\node [right] at (2,1) {$0$};
\end{tikzpicture}}
}
\subfigure[$5_{\rm I}$]
{{\scalefont{0.4}
\begin{tikzpicture}[scale=0.55]
\draw (1,0)--(-0.25,1);
\draw (1,0)--(0.5,1);
\draw (1,0)--(1.5,1);
\draw (1,0)--(2.17,0.92);
\draw [fill] (-0.25,1) circle [radius=0.1];
\draw [fill] (0.5,1) circle [radius=0.1];
\draw [fill] (1.5,1) circle [radius=0.1];
\draw [fill] (1,0) circle [radius=0.1];
\draw (2.25,1) circle [radius=0.1];
\node [right] at (1,0) {\textrm{-}$1$};
\node [right] at (-0.25,1) {$0$};
\node [right] at (1.5,1) {$0$};
\node [right] at (0.5,1) {$0$};
\node [right] at (2.25,1) {$1$};
\end{tikzpicture}}
}
\subfigure[$5_{\rm J}$]
{{\scalefont{0.4}
\begin{tikzpicture}[scale=0.55]
\draw (1,0)--(0.25,1)--(1,2)--(1,2.9);
\draw (1,0)--(1.75,1)--(1.08,1.92);
\draw [fill] (0.25,1) circle [radius=0.1];
\draw [fill] (1.75,1) circle [radius=0.1];
\draw [fill] (1,0) circle [radius=0.1];
\draw [fill] (1,2) circle [radius=0.1];
\draw (1,3) circle [radius=0.1];
\node [right] at (1,0) {$0$};
\node [right] at (0.25,1) {$0$};
\node [right] at (1.75,1) {$0$};
\node [right] at (1,2) {\textrm{-}$1$};
\node [right] at (1,3) {$1$};
\end{tikzpicture}}
}
\subfigure[$5_{\rm K}$]
{{\scalefont{0.4}
\begin{tikzpicture}[scale=0.55]
\draw (1,0)--(0.25,1)--(1,2);
\draw (1,0)--(1.75,1)--(1,2);
\draw (1.75,1)--(1.75,1.9);
\draw [fill] (0.25,1) circle [radius=0.1];
\draw [fill] (1.75,1) circle [radius=0.1];
\draw [fill] (1,0) circle [radius=0.1];
\draw (1.75,2) circle [radius=0.1];
\draw [fill] (1,2) circle [radius=0.1];
\node [right] at (1,0) {$0$};
\node [right] at (0.25,1) {$0$};
\node [right] at (1.75,1) {\textrm{-}$1$};
\node [right] at (1,2) {$0$};
\node [right] at (1.75,2) {$1$};
\end{tikzpicture}}
}
\subfigure[$5_{\rm L}$]
{{\scalefont{0.4}
\begin{tikzpicture}[scale=0.55]
\draw (1,0)--(0.25,1)--(0.92,1.92);
\draw (1,0)--(1,0.9);
\draw (1,0)--(1.75,1)--(1.08,1.92);
\draw [fill] (0.25,1) circle [radius=0.1];
\draw [fill] (1.75,1) circle [radius=0.1];
\draw [fill] (1,0) circle [radius=0.1];
\draw [fill] (1,2) circle [radius=0.1];
\draw (1,1) circle [radius=0.1];
\node [right] at (1,0) {\textrm{-}$1$};
\node [right] at (0.25,1) {$0$};
\node [right] at (1.75,1) {$0$};
\node [right] at (1,2) {$0$};
\node [right] at (1,1) {$1$};
\end{tikzpicture}}
}
\subfigure[$5_{\rm M}$]
{{\scalefont{0.4}
\begin{tikzpicture}[scale=0.55]
\draw (1,0)--(0.25,1)--(0.92,1.92);
\draw (1,0)--(1,-0.9);
\draw (1,0)--(1.75,1)--(1.08,1.92);
\draw [fill] (0.25,1) circle [radius=0.1];
\draw [fill] (1.75,1) circle [radius=0.1];
\draw [fill] (1,0) circle [radius=0.1];
\draw (1,2) circle [radius=0.1];
\draw [fill] (1,-1) circle [radius=0.1];
\node [right] at (1,0) {$1$};
\node [right] at (0.25,1) {\textrm{-}$1$};
\node [right] at (1.75,1) {\textrm{-}$1$};
\node [right] at (1,2) {$1$};
\node [right] at (1,-1) {$0$};
\end{tikzpicture}}
}
\subfigure[$5_{\rm N}$]
{{\scalefont{0.4}
\begin{tikzpicture}[scale=0.55]
\draw (1,0)--(0,0.75)--(0,1.5)--(0.9,2.2);
\draw (1,0)--(2,1.125)--(1.1,2.2);
\draw [fill] (0,0.75) circle [radius=0.1];
\draw [fill] (0,1.5) circle [radius=0.1];
\draw [fill] (2,1.125) circle [radius=0.1];
\draw [fill] (1,0) circle [radius=0.1];
\draw (1,2.25) circle [radius=0.1];
\node [right] at (1,0) {$1$};
\node [right] at (0,0.8) {$0$};
\node [right] at (0,1.5) {\textrm{-}$1$};
\node [right] at (1,2.25) {$1$};
\node [right] at (2,1.125) {\textrm{-}$1$};
\end{tikzpicture}}
}
\subfigure[$5_{\rm O}$]
{{\scalefont{0.4}
\begin{tikzpicture}[scale=0.55]
\draw (1,0)--(0,1)--(0.92,1.92);
\draw (1,0)--(1,1.9);
\draw (1,0)--(2,1)--(1.08,1.92);
\draw [fill] (0,1) circle [radius=0.1];
\draw [fill] (1,1) circle [radius=0.1];
\draw [fill] (2,1) circle [radius=0.1];
\draw [fill] (1,0) circle [radius=0.1];
\draw (1,2) circle [radius=0.1];
\node [right] at (1,0) {$2$};
\node [right] at (0,1) {\textrm{-}$1$};
\node [right] at (1,1) {\textrm{-}$1$};
\node [right] at (2,1) {\textrm{-}$1$};
\node [right] at (1,2) {$1$};
\end{tikzpicture}}
}
\centerline{\bf Figure 2.}
\end{figure} 

In each class the white point stands for the last added element. For each class we have also marked the value of $\mu_S(x_i,x_n)$ next to each element $x_i$. The calculation of $\mu_S(x_i,x_n)$ bases on $(\ref{eq:mu_Srec})$.

\begin{theorem}\label{th:4.1}
Let $S$ be a set with at most $5$ elements.
\begin{itemize}
\item[\rm(i)] If $S\in 1_{\rm A}$, then $[S]_f$ is always invertible (under the condition $f(x)\neq0$ for all $x\in P$).
\item[\rm(ii)] If $S\in 2_{\rm A}, 3_{\rm A}, 3_{\rm B}, 4_{\rm A}, 4_{\rm B}, 4_{\rm C}, 4_{\rm D}, 5_{\rm A}, 5_{\rm B}, \ldots, 5_{\rm I}$ and $f\in {\cal F}_1$, then $[S]_f$ is invertible.
\item[\rm(iii)] If $S\in 4_{\rm E}, 5_{\rm J},5_{\rm K},5_{\rm L},5_{\rm M},5_{\rm N}$ and $f\in {\cal F}_2$, then $[S]_f$ is invertible.
\item[\rm(iv)] If $S\in 5_{\rm O}={\cal S}_{3,5}$ and $f\in{\cal F}_1\cap {\cal G}_{3,5}$, then $[S]_f$ is invertible.
\end{itemize}
\end{theorem}

\begin{proof}
(i) The one element case is trivial. (ii) If $S$ belongs to one of the classes mentioned in part (ii), then $S$ can be constructed by (M$_{1,i}$) only and thus $f\in {\cal F}_1$ is a sufficient condition for the invertibility of $[S]_f$, see Definition \ref{de:F_k}. (iii) If $S$ belongs to the classes mentioned in (iii), then both (M$_{1,i}$) and (M$_{2,i}$) are needed and therefore $f\in {\cal F}_2$ is sufficient for the invertibility. (iv) If $S\in 5_{\rm O}$, then the conditions for the invertibility of $[S]_f$ follow from Theorem \ref{th:3.4}.
\end{proof}

\begin{corollary}\label{cor:4.1}
If $S$ is a meet closed set with at most $5$ elements and $f\in{\cal F}_2\cap{\cal G}_{3,5}$, then $[S]_f$ is invertible. In particular, if $S$ is a gcd-closed set with at most $5$ elements, then $[S]$ is invertible.
\end{corollary}

\begin{proof}
The first part is a direct consequence of Theorem \ref{th:4.1}. For the second part we just have to recall that $N\in{\cal F}_2\cap{\cal G}_{3,5}$ by Remark \ref{rem:B-L}.
\end{proof}

\subsection{Case \texorpdfstring{$n=6$}{n=6}}\label{subsect:4.6}

For $n\geq 6$ we change our procedure slightly, since there are $53$ classes of meet closed sets for $n=6$ and $222$ for $n=7$ (see e.g. \cite{HeiRei}, the number of meet semilattices with $n$ elements equals the number of lattices with $n+1$ elements, since adding a maximum element to a meet semilattice results a lattice). Here we construct only the meet closed sets with $6$ elements, where at least one of $m_1,\ldots,m_{n}$ is greater than or equal to $3$. (If  $m_1,\ldots,m_{n}\le 2$, then the Bourque-Ligh conjecture holds by Remark \ref{rem:B-L}.) We obtain exactly 7 different classes $6_{\rm A},6_{\rm B},\ldots,6_{\rm G}$ presented in Figure 3. In each class there can be no more than one element $x_i$ with $m_i\geq3$ and there exists exactly one class with $m_i=4$. Keeping this in mind the use of mathematical programs is not necessarily needed in order to find all 7 classes, but it would be easy to do so by making suitable adjustments to the code given in Remark \ref{re:calc}. The value of $\mu(x_i, x_6)$ is again marked next to each element $x_i$, and the white points stand for the last added element $x_6$.

\setcounter{subfigure}{0}
\begin{figure}[htb!]
\centering
\subfigure[$6_{\rm A}$]
{{\scalefont{0.4}
\begin{tikzpicture}[scale=0.55]
\draw (1,0)--(0.2,1)--(1,2)--(1,2.9);
\draw (1,0)--(1,2);
\draw (1,0)--(1.8,1)--(1,2);
\draw [fill] (0.2,1) circle [radius=0.1];
\draw [fill] (1,1) circle [radius=0.1];
\draw [fill] (1.8,1) circle [radius=0.1];
\draw [fill] (1,0) circle [radius=0.1];
\draw [fill] (1,2) circle [radius=0.1];
\draw (1,3) circle [radius=0.1];
\node [right] at (1,0) {$0$};
\node [right] at (0.2,1) {$0$};
\node [right] at (1,1) {$0$};
\node [right] at (1.8,1) {$0$};
\node [right] at (1,2) {\textrm{-}$1$};
\node [right] at (1,3) {$1$};
\end{tikzpicture}}
}
\subfigure[$6_{\rm B}$]
{{\scalefont{0.4}
\begin{tikzpicture}[scale=0.55]
\draw (1,0)--(0.2,1)--(1,2);
\draw (1,0)--(1,2);
\draw (0.2,1)--(0.2,1.9);
\draw (1,0)--(1.8,1)--(1,2);
\draw [fill] (0.2,1) circle [radius=0.1];
\draw [fill] (1,1) circle [radius=0.1];
\draw [fill] (1.8,1) circle [radius=0.1];
\draw [fill] (1,0) circle [radius=0.1];
\draw [fill] (1,2) circle [radius=0.1];
\draw (0.2,2) circle [radius=0.1];
\node [right] at (1,0) {$0$};
\node [right] at (1,2) {$0$};
\node [right] at (1,1) {$0$};
\node [right] at (1.8,1) {$0$};
\node [right] at (0.2,1) {\textrm{-}$1$};
\node [right] at (0.2,2) {$1$};
\end{tikzpicture}}
}
\subfigure[$6_{\rm C}$]
{{\scalefont{0.4}
\begin{tikzpicture}[scale=0.55]
\draw (1,0)--(0.2,1)--(1,2);
\draw (1,0)--(1,2);
\draw (1,0)--(1.9,0.295);
\draw (1,0)--(1.8,1)--(1,2);
\draw [fill] (0.2,1) circle [radius=0.1];
\draw [fill] (1,1) circle [radius=0.1];
\draw [fill] (1.8,1) circle [radius=0.1];
\draw [fill] (1,0) circle [radius=0.1];
\draw [fill] (1,2) circle [radius=0.1];
\draw (2,0.32) circle [radius=0.1];
\node [right] at (0.9,-0.15) {\textrm{-}$1$};
\node [right] at (0.2,1) {$0$};
\node [right] at (1,1) {$0$};
\node [right] at (1.8,1) {$0$};
\node [right] at (1,2) {$0$};
\node [right] at (2,0.35) {$1$};
\end{tikzpicture}}
}
\subfigure[$6_{\rm D}$]
{{\scalefont{0.4}
\begin{tikzpicture}[scale=0.55]
\draw (1,0)--(0.2,1)--(0.92,1.93);
\draw (1,-1)--(1,1.9);
\draw (1,0)--(1.8,1)--(1.08,1.93);
\draw [fill] (0.2,1) circle [radius=0.1];
\draw [fill] (1,1) circle [radius=0.1];
\draw [fill] (1.8,1) circle [radius=0.1];
\draw [fill] (1,0) circle [radius=0.1];
\draw (1,2) circle [radius=0.1];
\draw [fill] (1,-1) circle [radius=0.1];
\node [right] at (1,0) {$2$};
\node [right] at (0.2,1) {\textrm{-}$1$};
\node [right] at (1,1) {\textrm{-}$1$};
\node [right] at (1.8,1) {\textrm{-}$1$};
\node [right] at (1,2) {$1$};
\node [right] at (1,-1) {$0$};
\end{tikzpicture}}
}
\subfigure[$6_{\rm E}$]
{{\scalefont{0.4}
\begin{tikzpicture}[scale=0.55]
\draw (1,-1)--(0.2,1)--(0.92,1.93);
\draw (1,-1)--(1,1.9);
\draw (1,-1)--(1.8,1)--(1.08,1.93);
\draw [fill] (0.2,1) circle [radius=0.1];
\draw [fill] (1,1) circle [radius=0.1];
\draw [fill] (1.8,1) circle [radius=0.1];
\draw [fill] (1,0) circle [radius=0.1];
\draw (1,2) circle [radius=0.1];
\draw [fill] (1,-1) circle [radius=0.1];
\node [right] at (1,0.1) {$0$};
\node [right] at (0.2,1) {\textrm{-}$1$};
\node [right] at (1,1) {\textrm{-}$1$};
\node [right] at (1.8,1) {\textrm{-}$1$};
\node [right] at (1,2) {$1$};
\node [right] at (1,-1) {$2$};
\end{tikzpicture}}
}
\subfigure[$6_{\rm F}$]
{{\scalefont{0.4}
\begin{tikzpicture}[scale=0.55]
\draw (1,-1)--(0.2,1)--(0.92,1.93);
\draw (1,-1)--(1,1.9);
\draw (1,0)--(1.8,1)--(1.08,1.93);
\draw [fill] (0.2,1) circle [radius=0.1];
\draw [fill] (1,1) circle [radius=0.1];
\draw [fill] (1.8,1) circle [radius=0.1];
\draw [fill] (1,0) circle [radius=0.1];
\draw (1,2) circle [radius=0.1];
\draw [fill] (1,-1) circle [radius=0.1];
\node [right] at (1,0) {$1$};
\node [right] at (0.2,1) {\textrm{-}$1$};
\node [right] at (1,1) {\textrm{-}$1$};
\node [right] at (1.8,1) {\textrm{-}$1$};
\node [right] at (1,2) {$1$};
\node [right] at (1,-1) {$1$};
\end{tikzpicture}}
}
\subfigure[$6_{\rm G}$]
{{\scalefont{0.4}
\begin{tikzpicture}[scale=0.55]
\draw (1,0)--(-0.75,1)--(0.92,1.92);
\draw (1,0)--(0.2,1)--(0.92,1.92);
\draw (1,0)--(1.8,1)--(1.08,1.92);
\draw (1,0)--(2.75,1)--(1.08,1.92);
\draw [fill] (0.2,1) circle [radius=0.1];
\draw [fill] (1.8,1) circle [radius=0.1];
\draw [fill] (1,0) circle [radius=0.1];
\draw [fill] (2.75,1) circle [radius=0.1];
\draw [fill] (-0.75,1) circle [radius=0.1];
\draw (1,2) circle [radius=0.1];
\node [right] at (1,-0.05) {$3$};
\node [right] at (0.2,1) {\textrm{-}$1$};
\node [right] at (-0.75,1) {\textrm{-}$1$};
\node [right] at (2.75,1) {\textrm{-}$1$};
\node [right] at (1.8,1) {\textrm{-}$1$};
\node [right] at (1,2) {$1$};
\end{tikzpicture}}
}
\centerline{\bf Figure 3.}
\end{figure}

\begin{theorem}\label{th:4.4}
Let $S$ be a meet closed set with $6$ elements.
\begin{itemize}
\item[\rm(i)] If $S\not\in 6_{\rm A},6_{\rm B},\ldots,6_{\rm G}$ and $f\in{\cal F}_2$, then $[S]_f$ is invertible.
\item[\rm(ii)] If $S\in 6_{\rm A},6_{\rm B},\ldots,6_{\rm E}$ and $f\in{\cal F}_1\cap {\cal G}_{3, 5}$, then $[S]_f$ is invertible.
\item[\rm(iii)] If $S\in 6_{\rm F}={\cal S}_{3,6}$ and $f\in {\cal F}_1\cap{\cal G}_{3, 6}$, then $[S]_f$ is invertible.
\item[\rm(iv)] If $S\in 6_{\rm G}={\cal S}_{4,6}$ and $f\in{\cal F}_1\cap{\cal G}_{4, 6}$, then $[S]_f$ is invertible.
\end{itemize}
\end{theorem}

\begin{proof}
(i) If $S\not\in 6_{\rm A},6_{\rm B},\ldots,6_{\rm G}$, then only {\rm(M$_{1,i}$)} and {\rm(M$_{2,i}$)} have been used, and thus the condition $f\in{\cal F}_2$ assures that $[S]_f$ is invertible, see Definition \ref{de:F_k}. (ii) If $S\in 6_{\rm A},6_{\rm B},6_{\rm C}, 6_{\rm D}, 6_{\rm E}$, then $S$ can be constructed by (M$_{1,i}$) and (M$_{3,i}$), and thus the assumption $f\in{\cal F}_1$ together with $f\in{\cal G}_{3, 5}$ assures the fulfillment of conditions (C$_{1,i}$) and (C$_{3,i}$) and therefore the invertibility of $[S]_f$. For $S\in 6_{\rm A}, 6_{\rm B}, 6_{\rm C}$ the condition (C$_{3,5}$) is clearly implied by $f$ belonging to ${\cal G}_{3, 5}$, and also for $S\in 6_{\rm D}, 6_{\rm E}$ the condition (C$_{3,6}$) is implied by the assumption $f\in {\cal G}_{3, 5}$ due to the zeros of $\mu_S(x_i, x_6)$ in $6_{\rm D},6_{\rm E}$ of Figure~3. (iii) In the case when $S\in 6_{\rm F}$ the semilattice $S_{n-1}$ can be constructed by (M$_{1,i}$) and $S$ can be constructed by (M$_{3,6}$) from $S_{n-1}$. In this case the assumption $f\in{\cal F}_1$ quarantees that the conditions (C$_{1,i}$) hold, whereas $f\in{\cal G}_{3,6}$ implies that (C$_{3,6}$) holds. Thus $[S]_f$ is invertible. (iv) If $S\in 6_{\rm G}$, then the conditions for the invertibility of $[S]_f$ come from those in Theorem \ref{th:3.4}.
\end{proof}

\begin{corollary}\label{cor:4.4}
If $S$ is a meet closed set with $6$ elements and $f\in{\cal F}_2\cap {\cal G}_{3, 5}\cap {\cal G}_{3, 6}\cap {\cal G}_{4, 6}$, then $[S]_f$ is invertible. In particular, if $S$ is a gcd-closed set with $6$ elements, then $[S]$ is invertible.
\end{corollary}

\begin{proof}
The first part of this corollary is obvious, since ${\cal F}_2 \subseteq {\cal F}_1$. We only need to prove the second part. We already know that $N\in {\cal F}_2\cap {\cal G}_{3, 5}\cap {\cal G}_{4, 6}$ (Remark \ref{rem:B-L} and Corollary \ref{cor:3.2}), so it suffices to prove that $N\in {\cal G}_{3, 6}$. Let $S\in 6_{\rm F}$,
\[
x_1=\hbox{gcd}(x_2,x_3)=\hbox{gcd}(x_3,x_4)=\hbox{gcd}(x_3,x_5),
\]
$x_2=\hbox{gcd}(x_4,x_5)$ and $\hbox{lcm}(x_3,x_4,x_5)\mid x_6$. Thus $x_2=ax_1$, $x_3=bx_1$, $x_4=acx_1$, $x_5=adx_1$, where $a, b, c, d\geq2$ and
\[
\hbox{gcd}(a,b)=\hbox{gcd}(b,c)=\hbox{gcd}(b,d)=\hbox{gcd}(c,d)=1.
\]
Therefore at least one of the numbers $c$ and $d$ must be greater than or equal to $3$, from which it follows that $cd-c-d>0$. Clearly we also have $b-1>0$ and $x_1,x_6>0$ and thus we obtain
\begin{eqnarray}\label{eq:4.10}
&&\frac{1}{x_6}-\frac{1}{x_5}-\frac{1}{x_4}-\frac{1}{x_3}+\frac{1}{x_2}+\frac{1}{x_1}=\\
&&\qquad\qquad\quad=\;\frac{1}{x_6}+\frac{-bc-bd-acd+bcd+abcd}{abcdx_1} \nonumber\\
&&\qquad\qquad\quad=\;\frac{1}{x_6}+\frac{acd(b-1)+b(cd-d-c)}{abcdx_1}>0.\nonumber
\end{eqnarray}
This implies that $N\in {\cal G}_{3, 6}$.
\end{proof}

\subsection{Case \texorpdfstring{$n=7$}{n=7}}\label{subsect:4.7}

As in the case $n=6$, we consider only the meet closed sets with $7$ elements, where at least one of $m_1,\ldots,m_{n}$ is greater than or equal to $3$. There are exactly 47 such semilattices, which we divide into ten categories $7_{\rm A},7_{\rm B},\ldots,7_{\rm I}$ based on their structure, see Figures 4-8. As before, we have marked the value of $\mu_S(x_i,x_7)$ next to each element $x_i$, and the last added elements $x_7$ are denoted by white points.

\begin{remark}\label{re:calc}
In the case $n=6$ it is well possible to find all meet semilattices in Figure 3 without any computer calculations. As one might expect, in the case $n=7$ the task of finding all meet semilattices with at least one $m_i\geq3$ without any help from a computer becomes quite overwhelming. With Sage 5.10 this can easily be done by using the command
\begin{center}
\verb+P7=[p for p in Posets(7) if p.is_meet_semilattice() and+\\
\verb+max([len(p.lower_covers(q)) for q in p.list()]) >= 3]+.
\end{center}
With the command
\begin{center}
\verb+for p in P7: show(p.plot())+
\end{center}
it is then possible to obtain the list of Hasse diagrams of the meet semilattices in question.
\end{remark}

\begin{theorem}\label{th:4.5}
Let $S$ be a meet closed set with $7$ elements.
\begin{itemize}
\item[\rm(i)] If $S$ does not belong to any classes presented in Figures 4-8 and $f\in{\cal F}_2$, then $[S]_f$ is invertible.
\item[\rm(ii)] If $S\in 7_{\rm AA},7_{\rm AB},\ldots,7_{\rm AX}$ and $f\in{\cal F}_1\cap {\cal G}_{3, 5}$, then $[S]_f$ is invertible.
\item[\rm(iii)] If $S\in 7_{\rm BA},7_{\rm BB},\ldots,7_{\rm BI}$ and $f\in {\cal F}_1\cap {\cal G}_{3, 6}$, then $[S]_f$ is invertible.
\item[\rm(iv)] If $S\in 7_{\rm CA},7_{\rm CB},7_{\rm CC},7_{\rm CD},7_{\rm CE}$ and $f\in{\cal F}_2\cap {\cal G}_{3, 5}$, then $[S]_f$ is invertible.
\item[\rm(v)] If $S\in 7_{\rm DA},7_{\rm DB},7_{\rm DC},7_{\rm DD},7_{\rm DE}$ and $f\in{\cal F}_1\cap {\cal G}_{4, 6}$, then $[S]_f$ is invertible.
\item[\rm(vi)] If $S\in 7_{\rm E}$ and $f\in {\cal F}_2\cap{\cal G}_{3, 6}$, then $[S]_f$ is invertible.
\item[\rm(vii)] If $S\in 7_{\rm F}$ and $f\in {\cal F}_2\cap{\cal G}_{3, 7}$, then $[S]_f$ is invertible.
\item[\rm(viii)] If $S\in 7_{\rm G}$ and $f\in {\cal F}_1\cap{\cal G}_{4, 7}^{(1)}$, then $[S]_f$ is invertible.
\item[\rm(ix)] If $S\in 7_{\rm H}$ and $f\in {\cal F}_1\cap{\cal G}_{4, 7}^{(2)}$, then $[S]_f$ is invertible.
\item[\rm(x)] If $S\in 7_{\rm I}$ and $f\in{\cal F}_1\cap {\cal G}_{5, 7}$, then $[S]_f$ is invertible.
\end{itemize}
\end{theorem}

\setcounter{subfigure}{0}
\begin{figure}[ht!]
\centering
\subfigure[$7_{\rm AA}$]
{{\scalefont{0.4}
\begin{tikzpicture}[scale=0.55]
\draw (1,0)--(0.2,1)--(1,2)--(1,2.7);
\draw (1,0)--(1,2);
\draw (1,0)--(1.8,1)--(1,2);
\draw (1,0)--(1.9,0.295);
\draw [fill] (0.2,1) circle [radius=0.1];
\draw [fill] (1,1) circle [radius=0.1];
\draw [fill] (1.8,1) circle [radius=0.1];
\draw [fill] (1,0) circle [radius=0.1];
\draw [fill] (1,2) circle [radius=0.1];
\draw (2,0.32) circle [radius=0.1];
\draw [fill] (1,2.8) circle [radius=0.1];
\node [right] at (1,-0.1) {\textrm{-}$1$};
\node [right] at (0.2,1) {$0$};
\node [right] at (1,1) {$0$};
\node [right] at (1.8,1) {$0$};
\node [right] at (1,2) {$0$};
\node [right] at (1,2.8) {$0$};
\node [right] at (2,0.35) {$1$};
\end{tikzpicture}}
}
\subfigure[$7_{\rm AB}$]
{{\scalefont{0.4}
\begin{tikzpicture}[scale=0.55]
\draw (1,0)--(0.2,1)--(1,2)--(1,2.7);
\draw (1,0)--(1,2);
\draw (1,0)--(1.8,1)--(1,2);
\draw (0.2,1)--(0.2,1.9);
\draw [fill] (0.2,1) circle [radius=0.1];
\draw [fill] (1,1) circle [radius=0.1];
\draw [fill] (1.8,1) circle [radius=0.1];
\draw [fill] (1,0) circle [radius=0.1];
\draw [fill] (1,2) circle [radius=0.1];
\draw (0.2,2) circle [radius=0.1];
\draw [fill] (1,2.8) circle [radius=0.1];
\node [right] at (1,-0.1) {$0$};
\node [right] at (0.2,1) {\textrm{-}$1$};
\node [right] at (1,1) {$0$};
\node [right] at (1.8,1) {$0$};
\node [right] at (1,2) {$0$};
\node [right] at (1,2.8) {$0$};
\node [right] at (0.2,2) {$1$};
\end{tikzpicture}}
}
\subfigure[$7_{\rm AC}$]
{{\scalefont{0.4}
\begin{tikzpicture}[scale=0.55]
\draw (1,0)--(0.2,1)--(1,2)--(0.5,2.7);
\draw (1,0)--(1,2)--(1.45,2.62);
\draw (1,0)--(1.8,1)--(1,2);
\draw [fill] (0.2,1) circle [radius=0.1];
\draw [fill] (1,1) circle [radius=0.1];
\draw [fill] (1.8,1) circle [radius=0.1];
\draw [fill] (1,0) circle [radius=0.1];
\draw [fill] (1,2) circle [radius=0.1];
\draw [fill] (0.5,2.7) circle [radius=0.1];
\draw (1.5,2.7) circle [radius=0.1];
\node [right] at (1,0) {$0$};
\node [right] at (0.2,1) {$0$};
\node [right] at (1,1) {$0$};
\node [right] at (1.8,1) {$0$};
\node [right] at (1,2) {\textrm{-}$1$};
\node [right] at (0.5,2.7) {$0$};
\node [right] at (1.5,2.7) {$1$};
\end{tikzpicture}}
}
\subfigure[$7_{\rm AD}$]
{{\scalefont{0.4}
\begin{tikzpicture}[scale=0.55]
\draw (1,0)--(0.2,1)--(1,2)--(1,2.8);
\draw (1,0)--(1,2);
\draw (1,0)--(1.8,1)--(1,2);
\draw (1,2.8)--(1.9,3.095);
\draw (2,3.12) circle [radius=0.1];
\draw [fill] (0.2,1) circle [radius=0.1];
\draw [fill] (1,1) circle [radius=0.1];
\draw [fill] (1.8,1) circle [radius=0.1];
\draw [fill] (1,0) circle [radius=0.1];
\draw [fill] (1,2) circle [radius=0.1];
\draw [fill] (1,2.8) circle [radius=0.1];
\node [right] at (1,0) {$0$};
\node [right] at (0.2,1) {$0$};
\node [right] at (1,1) {$0$};
\node [right] at (1.8,1) {$0$};
\node [right] at (1,2) {$0$};
\node [right] at (1,2.7) {\textrm{-}$1$};
\node [right] at (2,3.12) {$1$};
\end{tikzpicture}}
}
\subfigure[$7_{\rm AE}$]
{{\scalefont{0.4}
\begin{tikzpicture}[scale=0.55]
\draw (1,0)--(0.2,1)--(1,2);
\draw (1,0)--(1,2);
\draw (0.2,1)--(0.2,1.9);
\draw (1,0)--(1.8,1)--(1,2);
\draw (1,0)--(1.9,0.295);
\draw [fill] (0.2,1) circle [radius=0.1];
\draw [fill] (1,1) circle [radius=0.1];
\draw [fill] (1.8,1) circle [radius=0.1];
\draw [fill] (1,0) circle [radius=0.1];
\draw [fill] (1,2) circle [radius=0.1];
\draw [fill] (0.2,2) circle [radius=0.1];
\draw (2,0.32) circle [radius=0.1];
\node [right] at (1,-0.1) {\textrm{-}$1$};
\node [right] at (1,2) {$0$};
\node [right] at (1,1) {$0$};
\node [right] at (1.8,1) {$0$};
\node [right] at (0.2,1) {$0$};
\node [right] at (0.2,2) {$0$};
\node [right] at (2,0.35) {$1$};
\end{tikzpicture}}
}
\subfigure[$7_{\rm AF}$]
{{\scalefont{0.4}
\begin{tikzpicture}[scale=0.55]
\draw (1,0)--(0.2,1)--(1,2);
\draw (1,0)--(1,2);
\draw (1,0)--(1.8,1)--(1,2);
\draw (0.2,1)--(0.2,1.9);
\draw (1.8,1)--(1.8,2);
\draw [fill] (0.2,1) circle [radius=0.1];
\draw [fill] (1,1) circle [radius=0.1];
\draw [fill] (1.8,1) circle [radius=0.1];
\draw [fill] (1,0) circle [radius=0.1];
\draw [fill] (1,2) circle [radius=0.1];
\draw (0.2,2) circle [radius=0.1];
\draw [fill] (1.8,2) circle [radius=0.1];
\node [right] at (1,-0.1) {$0$};
\node [right] at (0.2,1) {\textrm{-}$1$};
\node [right] at (1,1) {$0$};
\node [right] at (1.8,1) {$0$};
\node [right] at (1,2) {$0$};
\node [right] at (1.8,2) {$0$};
\node [right] at (0.2,2) {$1$};
\end{tikzpicture}}
}
\subfigure[$7_{\rm AG}$]
{{\scalefont{0.4}
\begin{tikzpicture}[scale=0.55]
\draw (1,0)--(0.2,1)--(1,2);
\draw (1,0)--(1,2);
\draw (1,0)--(1.8,1)--(1,2);
\draw (0.2,1)--(0.2,1.9);
\draw (0.2,1)--(-0.6,1.9);
\draw [fill] (0.2,1) circle [radius=0.1];
\draw [fill] (1,1) circle [radius=0.1];
\draw [fill] (1.8,1) circle [radius=0.1];
\draw [fill] (1,0) circle [radius=0.1];
\draw [fill] (1,2) circle [radius=0.1];
\draw (-0.6,2) circle [radius=0.1];
\draw [fill] (0.2,2) circle [radius=0.1];
\node [right] at (1,-0.1) {$0$};
\node [right] at (0.2,1) {\textrm{-}$1$};
\node [right] at (1,1) {$0$};
\node [right] at (1.8,1) {$0$};
\node [right] at (1,2) {$0$};
\node [right] at (0.2,2) {$0$};
\node [right] at (-0.6,2) {$1$};
\end{tikzpicture}}
}
\subfigure[$7_{\rm AH}$]
{{\scalefont{0.4}
\begin{tikzpicture}[scale=0.55]
\draw (1,0)--(0.2,1)--(1,2);
\draw (1,0)--(1,2);
\draw (0.2,1)--(0.2,2.9);
\draw (1,0)--(1.8,1)--(1,2);
\draw [fill] (0.2,1) circle [radius=0.1];
\draw [fill] (1,1) circle [radius=0.1];
\draw [fill] (1.8,1) circle [radius=0.1];
\draw [fill] (1,0) circle [radius=0.1];
\draw [fill] (1,2) circle [radius=0.1];
\draw [fill] (0.2,2) circle [radius=0.1];
\draw (0.2,3) circle [radius=0.1];
\node [right] at (0.2,3) {$1$};
\node [right] at (1,0) {$0$};
\node [right] at (1,2) {$0$};
\node [right] at (1,1) {$0$};
\node [right] at (1.8,1) {$0$};
\node [right] at (0.2,1) {$0$};
\node [right] at (0.2,2) {\textrm{-}$1$};
\end{tikzpicture}}
}
\subfigure[$7_{\rm AI}$]
{{\scalefont{0.4}
\begin{tikzpicture}[scale=0.55]
\draw (2,0)--(1,-0.5)--(0.08,-0.05);
\draw (1,-0.5)--(1,1)--(1,1.9);
\draw (1,-0.5)--(2,1)--(1.08,1.93);
\draw (1,-0.5)--(0,1)--(0.92,1.93);
\draw [fill] (2,0) circle [radius=0.1];
\draw [fill] (1,-0.5) circle [radius=0.1];
\draw (0,0) circle [radius=0.1];
\draw [fill] (0,1) circle [radius=0.1];
\draw [fill] (1,1) circle [radius=0.1];
\draw [fill] (1,2) circle [radius=0.1];
\draw [fill] (2,1) circle [radius=0.1];
\node [right] at (2,0) {$0$};
\node [right] at (1,-0.55) {\textrm{-}$1$};
\node [right] at (0,0.1) {$1$};
\node [right] at (0,1) {$0$};
\node [right] at (1,1) {$0$};
\node [right] at (2,1) {$0$};
\node [right] at (1,2) {$0$};
\end{tikzpicture}}
}
\subfigure[$7_{\rm AJ}$]
{{\scalefont{0.4}
\begin{tikzpicture}[scale=0.55]
\draw (2,-1)--(0.08,-0.05);
\draw (2,-1)--(1,1)--(1,1.9);
\draw (2,-1)--(2,1)--(1.08,1.93);
\draw (2,-1)--(0,1)--(0.92,1.93);
\draw [fill] (2,-1) circle [radius=0.1];
\draw [fill] (1,-0.5) circle [radius=0.1];
\draw (0,0) circle [radius=0.1];
\draw [fill] (0,1) circle [radius=0.1];
\draw [fill] (1,1) circle [radius=0.1];
\draw [fill] (1,2) circle [radius=0.1];
\draw [fill] (2,1) circle [radius=0.1];
\node [right] at (2,-1) {$0$};
\node [right] at (0.8,-0.8) {\textrm{-}$1$};
\node [right] at (0,0.1) {$1$};
\node [right] at (0,1) {$0$};
\node [right] at (1,1) {$0$};
\node [right] at (2,1) {$0$};
\node [right] at (1,2) {$0$};
\end{tikzpicture}}
}
\subfigure[$7_{\rm AK}$]
{{\scalefont{0.4}
\begin{tikzpicture}[scale=0.55]
\draw (1.92,-0.05)--(1,-0.5)--(0,0)--(0,1)--(0.92,1.93);
\draw (0,0)--(1,1)--(1,1.9);
\draw (0,0)--(2,1)--(1.08,1.93);
\draw (2,0) circle [radius=0.1];
\draw [fill] (1,-0.5) circle [radius=0.1];
\draw [fill] (0,0) circle [radius=0.1];
\draw [fill] (0,1) circle [radius=0.1];
\draw [fill] (1,1) circle [radius=0.1];
\draw [fill] (1,2) circle [radius=0.1];
\draw [fill] (2,1) circle [radius=0.1];
\node [right] at (2,0) {$1$};
\node [right] at (1,-0.55) {\textrm{-}$1$};
\node [right] at (0.1,0) {$0$};
\node [right] at (0,1) {$0$};
\node [right] at (1,1) {$0$};
\node [right] at (2,1) {$0$};
\node [right] at (1,2) {$0$};
\end{tikzpicture}}
}
\subfigure[$7_{\rm AL}$]
{{\scalefont{0.4}
\begin{tikzpicture}[scale=0.55]
\draw (2,-1)--(0.08,-0.05);
\draw (1,-0.5)--(1,1)--(1,1.9);
\draw (1,-0.5)--(2,1)--(1.08,1.93);
\draw (1,-0.5)--(0,1)--(0.92,1.93);
\draw [fill] (2,-1) circle [radius=0.1];
\draw [fill] (1,-0.5) circle [radius=0.1];
\draw (0,0) circle [radius=0.1];
\draw [fill] (0,1) circle [radius=0.1];
\draw [fill] (1,1) circle [radius=0.1];
\draw [fill] (1,2) circle [radius=0.1];
\draw [fill] (2,1) circle [radius=0.1];
\node [right] at (2,-1) {$0$};
\node [right] at (1.05,-0.45) {\textrm{-}$1$};
\node [right] at (0,0.1) {$1$};
\node [right] at (0,1) {$0$};
\node [right] at (1,1) {$0$};
\node [right] at (2,1) {$0$};
\node [right] at (1,2) {$0$};
\end{tikzpicture}}
}
\subfigure[$7_{\rm AM}$]
{{\scalefont{0.4}
\begin{tikzpicture}[scale=0.55]
\draw (1,0)--(0.2,1)--(0.92,1.93);
\draw (1,-1)--(1,1.9);
\draw (1,0)--(1.8,1)--(1.08,1.93);
\draw (0.2,1)--(0.2,1.9);
\draw (0.2,2) circle [radius=0.1];
\draw [fill] (0.2,1) circle [radius=0.1];
\draw [fill] (1,1) circle [radius=0.1];
\draw [fill] (1.8,1) circle [radius=0.1];
\draw [fill] (1,0) circle [radius=0.1];
\draw [fill] (1,2) circle [radius=0.1];
\draw [fill] (1,-1) circle [radius=0.1];
\node [right] at (1,0) {$0$};
\node [right] at (0.2,1) {\textrm{-}$1$};
\node [right] at (1,1) {$0$};
\node [right] at (1.8,1) {$0$};
\node [right] at (1,2) {$0$};
\node [right] at (1,-1) {$0$};
\node [right] at (0.2,2) {$1$};
\end{tikzpicture}}
}
\subfigure[$7_{\rm AN}$]
{{\scalefont{0.4}
\begin{tikzpicture}[scale=0.55]
\draw (1,0)--(0.2,1)--(0.92,1.93);
\draw (1,-1)--(1,1.9);
\draw (1,0)--(1.8,1)--(1.08,1.93);
\draw (1,2)--(1.9,2.295);
\draw (2,2.32) circle [radius=0.1];
\draw [fill] (0.2,1) circle [radius=0.1];
\draw [fill] (1,1) circle [radius=0.1];
\draw [fill] (1.8,1) circle [radius=0.1];
\draw [fill] (1,0) circle [radius=0.1];
\draw [fill] (1,2) circle [radius=0.1];
\draw [fill] (1,-1) circle [radius=0.1];
\node [right] at (1,0) {$0$};
\node [right] at (0.2,1) {$0$};
\node [right] at (1,1) {$0$};
\node [right] at (1.8,1) {$0$};
\node [right] at (1.05,1.95) {\textrm{-}$1$};
\node [right] at (1,-1) {$0$};
\node [right] at (2,2.32) {$1$};
\end{tikzpicture}}
}
\subfigure[$7_{\rm AO}$]
{{\scalefont{0.4}
\begin{tikzpicture}[scale=0.55]
\draw (1,-1)--(0.2,1)--(0.92,1.93);
\draw (1,-1)--(1,1.9);
\draw (1,-1)--(1.9,-0.705);
\draw (1,-1)--(1.8,1)--(1.08,1.93);
\draw (2,-0.68) circle [radius=0.1];
\draw [fill] (0.2,1) circle [radius=0.1];
\draw [fill] (1,1) circle [radius=0.1];
\draw [fill] (1.8,1) circle [radius=0.1];
\draw [fill] (1,0) circle [radius=0.1];
\draw [fill] (1,2) circle [radius=0.1];
\draw [fill] (1,-1) circle [radius=0.1];
\node [right] at (1,0.1) {$0$};
\node [right] at (0.2,1) {$0$};
\node [right] at (1,1) {$0$};
\node [right] at (1.8,1) {$0$};
\node [right] at (1,2) {$0$};
\node [right] at (1,-1.05) {\textrm{-}$1$};
\node [right] at (2,-0.65) {$1$};
\end{tikzpicture}}
}
\subfigure[$7_{\rm AP}$]
{{\scalefont{0.4}
\begin{tikzpicture}[scale=0.55]
\draw (1,-1)--(0.2,1)--(0.92,1.93);
\draw (1,-1)--(1,1.9);
\draw (1,-1)--(1.8,1)--(1.08,1.93);
\draw (0.2,1.9)--(0.2,1);
\draw (0.2,2) circle [radius=0.1];
\draw [fill] (0.2,1) circle [radius=0.1];
\draw [fill] (1,1) circle [radius=0.1];
\draw [fill] (1.8,1) circle [radius=0.1];
\draw [fill] (1,0) circle [radius=0.1];
\draw [fill] (1,2) circle [radius=0.1];
\draw [fill] (1,-1) circle [radius=0.1];
\node [right] at (1,0.1) {$0$};
\node [right] at (0.2,1) {\textrm{-}$1$};
\node [right] at (1,1) {$0$};
\node [right] at (1.8,1) {$0$};
\node [right] at (1,2) {$0$};
\node [right] at (1,-1) {$0$};
\node [right] at (0.2,2) {$1$};
\end{tikzpicture}}
}
\subfigure[$7_{\rm AQ}$]
{{\scalefont{0.4}
\begin{tikzpicture}[scale=0.55]
\draw (2,-1)--(0.08,-0.05);
\draw (2,-1)--(1,1)--(1,1.9);
\draw (2,-1)--(2,1)--(1.08,1.93);
\draw (1,-0.5)--(0,1)--(0.92,1.93);
\draw [fill] (2,-1) circle [radius=0.1];
\draw [fill] (1,-0.5) circle [radius=0.1];
\draw (0,0) circle [radius=0.1];
\draw [fill] (0,1) circle [radius=0.1];
\draw [fill] (1,1) circle [radius=0.1];
\draw [fill] (1,2) circle [radius=0.1];
\draw [fill] (2,1) circle [radius=0.1];
\node [right] at (2,-1) {$0$};
\node [right] at (1,-0.4) {\textrm{-}$1$};
\node [right] at (0,0.1) {$1$};
\node [right] at (0,1) {$0$};
\node [right] at (1,1) {$0$};
\node [right] at (2,1) {$0$};
\node [right] at (1,2) {$0$};
\end{tikzpicture}}
}
\subfigure[$7_{\rm AR}$]
{{\scalefont{0.4}
\begin{tikzpicture}[scale=0.55]
\draw (0,1)--(0,1.9);
\draw (1,-0.5)--(2,-1);
\draw (2,-1)--(1,1)--(1,1.9);
\draw (2,-1)--(2,1)--(1.08,1.93);
\draw (1,-0.5)--(0,1)--(0.92,1.93);
\draw [fill] (2,-1) circle [radius=0.1];
\draw [fill] (1,-0.5) circle [radius=0.1];
\draw (0,2) circle [radius=0.1];
\draw [fill] (0,1) circle [radius=0.1];
\draw [fill] (1,1) circle [radius=0.1];
\draw [fill] (1,2) circle [radius=0.1];
\draw [fill] (2,1) circle [radius=0.1];
\node [right] at (2,-1) {$0$};
\node [right] at (1,-0.4) {$0$};
\node [right] at (0,2) {$1$};
\node [right] at (0,1) {\textrm{-}$1$};
\node [right] at (1,1) {$0$};
\node [right] at (2,1) {$0$};
\node [right] at (1,2) {$0$};
\end{tikzpicture}}
}
\subfigure[$7_{\rm AS}$]
{{\scalefont{0.4}
\begin{tikzpicture}[scale=0.55]
\draw (1,-1)--(0.2,1)--(0.92,1.93);
\draw (1,-1)--(1,1.9);
\draw (1,-1)--(1.8,1)--(1.08,1.93);
\draw (1,2)--(1.9,2.295);
\draw (2,2.32) circle [radius=0.1];
\draw [fill] (0.2,1) circle [radius=0.1];
\draw [fill] (1,1) circle [radius=0.1];
\draw [fill] (1.8,1) circle [radius=0.1];
\draw [fill] (1,0) circle [radius=0.1];
\draw [fill] (1,2) circle [radius=0.1];
\draw [fill] (1,-1) circle [radius=0.1];
\node [right] at (1,0.1) {$0$};
\node [right] at (0.2,1) {$0$};
\node [right] at (1,1) {$0$};
\node [right] at (1.8,1) {$0$};
\node [right] at (1.05,1.95) {\textrm{-}$1$};
\node [right] at (1,-1) {$0$};
\node [right] at (2,2.32) {$1$};
\end{tikzpicture}}
}
\subfigure[$7_{\rm AT}$]
{{\scalefont{0.4}
\begin{tikzpicture}[scale=0.55]
\draw (2,-1)--(0,0)--(0,1)--(0.92,1.93);
\draw (2,-1)--(1,1)--(1,1.9);
\draw (2,-1)--(2,1)--(1.08,1.93);
\draw [fill] (2,-1) circle [radius=0.1];
\draw [fill] (1,-0.5) circle [radius=0.1];
\draw [fill] (0,0) circle [radius=0.1];
\draw [fill] (0,1) circle [radius=0.1];
\draw [fill] (1,1) circle [radius=0.1];
\draw (1,2) circle [radius=0.1];
\draw [fill] (2,1) circle [radius=0.1];
\node [right] at (2,-1) {$2$};
\node [right] at (1,-0.4) {$0$};
\node [right] at (0,0.1) {$0$};
\node [right] at (0,1) {\textrm{-}$1$};
\node [right] at (1,1) {\textrm{-}$1$};
\node [right] at (2,1) {\textrm{-}$1$};
\node [right] at (1,2) {$1$};
\end{tikzpicture}}
}
\subfigure[$7_{\rm AU}$]
{{\scalefont{0.4}
\begin{tikzpicture}[scale=0.55]
\draw (2,-1)--(0,0)--(0,1)--(0.92,1.93);
\draw (1,-0.5)--(1,1)--(1,1.9);
\draw (1,-0.5)--(2,1)--(1.08,1.93);
\draw [fill] (2,-1) circle [radius=0.1];
\draw [fill] (1,-0.5) circle [radius=0.1];
\draw [fill] (0,0) circle [radius=0.1];
\draw [fill] (0,1) circle [radius=0.1];
\draw [fill] (1,1) circle [radius=0.1];
\draw (1,2) circle [radius=0.1];
\draw [fill] (2,1) circle [radius=0.1];
\node [right] at (2,-1) {$0$};
\node [right] at (1.05,-0.45) {$2$};
\node [right] at (0,0.1) {$0$};
\node [right] at (0,1) {\textrm{-}$1$};
\node [right] at (1,1) {\textrm{-}$1$};
\node [right] at (2,1) {\textrm{-}$1$};
\node [right] at (1,2) {$1$};
\end{tikzpicture}}
}
\subfigure[$7_{\rm AV}$]
{{\scalefont{0.4}
\begin{tikzpicture}[scale=0.55]
\draw (2,-1)--(0,0)--(0,1)--(0.92,1.93);
\draw (0,0)--(1,1)--(1,1.9);
\draw (0,0)--(2,1)--(1.08,1.93);
\draw [fill] (2,-1) circle [radius=0.1];
\draw [fill] (1,-0.5) circle [radius=0.1];
\draw [fill] (0,0) circle [radius=0.1];
\draw [fill] (0,1) circle [radius=0.1];
\draw [fill] (1,1) circle [radius=0.1];
\draw (1,2) circle [radius=0.1];
\draw [fill] (2,1) circle [radius=0.1];
\node [right] at (2,-1) {$0$};
\node [right] at (1,-0.4) {$0$};
\node [right] at (0,0) {\ $2$};
\node [right] at (0,1) {\textrm{-}$1$};
\node [right] at (1,1) {\textrm{-}$1$};
\node [right] at (2,1) {\textrm{-}$1$};
\node [right] at (1,2) {$1$};
\end{tikzpicture}}
}
\subfigure[$7_{\rm AX}$]
{{\scalefont{0.4}
\begin{tikzpicture}[scale=0.55]
\draw (2,0)--(1,-0.5)--(0,0)--(0,1)--(0.92,1.93);
\draw (1,-0.5)--(1,1)--(1,1.9);
\draw (2,0)--(2,1)--(1.08,1.93);
\draw [fill] (2,0) circle [radius=0.1];
\draw [fill] (1,-0.5) circle [radius=0.1];
\draw [fill] (0,0) circle [radius=0.1];
\draw [fill] (0,1) circle [radius=0.1];
\draw [fill] (1,1) circle [radius=0.1];
\draw (1,2) circle [radius=0.1];
\draw [fill] (2,1) circle [radius=0.1];
\node [right] at (2,0) {$0$};
\node [right] at (1,-0.6) {$2$};
\node [right] at (0,0.1) {$0$};
\node [right] at (0,1) {\textrm{-}$1$};
\node [right] at (1,1) {\textrm{-}$1$};
\node [right] at (2,1) {\textrm{-}$1$};
\node [right] at (1,2) {$1$};
\end{tikzpicture}}
}
\centerline{\bf Figure 4.}
\end{figure}
\setcounter{subfigure}{0}
\begin{figure}[ht!]
\centering
\subfigure[$7_{\rm BA}$]
{{\scalefont{0.4}
\begin{tikzpicture}[scale=0.55]
\draw (1.92,-0.05)--(1,-0.5)--(0,0)--(0,1)--(0.92,1.93);
\draw (0,0)--(1,1)--(1,1.9);
\draw (1,-0.5)--(2,1)--(1.08,1.93);
\draw (2,0) circle [radius=0.1];
\draw [fill] (1,-0.5) circle [radius=0.1];
\draw [fill] (0,0) circle [radius=0.1];
\draw [fill] (0,1) circle [radius=0.1];
\draw [fill] (1,1) circle [radius=0.1];
\draw [fill] (1,2) circle [radius=0.1];
\draw [fill] (2,1) circle [radius=0.1];
\node [right] at (2,0) {$1$};
\node [right] at (1,-0.6) {\textrm{-}$1$};
\node [right] at (0.05,0.05) {$0$};
\node [right] at (0,1) {$0$};
\node [right] at (1,1) {$0$};
\node [right] at (2,1) {$0$};
\node [right] at (1,2) {$0$};
\end{tikzpicture}}
}
\subfigure[$7_{\rm BB}$]
{{\scalefont{0.4}
\begin{tikzpicture}[scale=0.55]
\draw (2,-1)--(0.08,-0.05);
\draw (1,-0.5)--(1,1)--(1,1.9);
\draw (2,-1)--(2,1)--(1.08,1.93);
\draw (1,-0.5)--(0,1)--(0.92,1.93);
\draw [fill] (2,-1) circle [radius=0.1];
\draw [fill] (1,-0.5) circle [radius=0.1];
\draw (0,0) circle [radius=0.1];
\draw [fill] (0,1) circle [radius=0.1];
\draw [fill] (1,1) circle [radius=0.1];
\draw [fill] (1,2) circle [radius=0.1];
\draw [fill] (2,1) circle [radius=0.1];
\node [right] at (2,-1) {$0$};
\node [right] at (1,-0.4) {\textrm{-}$1$};
\node [right] at (0,0.1) {$1$};
\node [right] at (0,1) {$0$};
\node [right] at (1,1) {$0$};
\node [right] at (2,1) {$0$};
\node [right] at (1,2) {$0$};
\end{tikzpicture}}
}
\subfigure[$7_{\rm BC}$]
{{\scalefont{0.4}
\begin{tikzpicture}[scale=0.55]
\draw (1,-1)--(0.2,1)--(0.92,1.93);
\draw (0.2,1.9)--(0.2,1);
\draw (1,-1)--(1,1.9);
\draw (1,0)--(1.8,1)--(1.08,1.93);
\draw [fill] (0.2,1) circle [radius=0.1];
\draw [fill] (1,1) circle [radius=0.1];
\draw [fill] (1.8,1) circle [radius=0.1];
\draw [fill] (1,0) circle [radius=0.1];
\draw [fill] (1,2) circle [radius=0.1];
\draw (0.2,2) circle [radius=0.1];
\draw [fill] (1,-1) circle [radius=0.1];
\node [right] at (1,0) {$0$};
\node [right] at (0.2,1) {\textrm{-}$1$};
\node [right] at (1,1) {$0$};
\node [right] at (1.8,1) {$0$};
\node [right] at (1,2) {$0$};
\node [right] at (1,-1) {$0$};
\node [right] at (0.2,2) {$1$};
\end{tikzpicture}}
}
\subfigure[$7_{\rm BD}$]
{{\scalefont{0.4}
\begin{tikzpicture}[scale=0.55]
\draw (1,-1)--(0.2,1)--(0.92,1.93);
\draw (1.8,1.9)--(1.8,1);
\draw (1,-1)--(1,1.9);
\draw (1,0)--(1.8,1)--(1.08,1.93);
\draw [fill] (0.2,1) circle [radius=0.1];
\draw [fill] (1,1) circle [radius=0.1];
\draw [fill] (1.8,1) circle [radius=0.1];
\draw [fill] (1,0) circle [radius=0.1];
\draw [fill] (1,2) circle [radius=0.1];
\draw (1.8,2) circle [radius=0.1];
\draw [fill] (1,-1) circle [radius=0.1];
\node [right] at (1,0) {$0$};
\node [right] at (1.8,1) {\textrm{-}$1$};
\node [right] at (1,1) {$0$};
\node [right] at (0.2,1) {$0$};
\node [right] at (1,2) {$0$};
\node [right] at (1,-1) {$0$};
\node [right] at (1.8,2) {$1$};
\end{tikzpicture}}
}
\subfigure[$7_{\rm BE}$]
{{\scalefont{0.4}
\begin{tikzpicture}[scale=0.55]
\draw (1,-1)--(0.2,1)--(0.92,1.93);
\draw (1,-1)--(1,1.9);
\draw (1,0)--(1.8,1)--(1.08,1.93);
\draw (1,0)--(1.8,1)--(1,2);
\draw (1,2)--(1.9,2.295);
\draw (2,2.32) circle [radius=0.1];
\draw [fill] (0.2,1) circle [radius=0.1];
\draw [fill] (1,1) circle [radius=0.1];
\draw [fill] (1.8,1) circle [radius=0.1];
\draw [fill] (1,0) circle [radius=0.1];
\draw [fill] (1,2) circle [radius=0.1];
\draw [fill] (1,-1) circle [radius=0.1];
\node [right] at (1,0) {$0$};
\node [right] at (0.2,1) {$0$};
\node [right] at (1,1) {$0$};
\node [right] at (1.8,1) {$0$};
\node [right] at (1.05,1.95) {\textrm{-}$1$};
\node [right] at (1,-1) {$0$};
\node [right] at (2,2.32) {$1$};
\end{tikzpicture}}
}
\subfigure[$7_{\rm BF}$]
{{\scalefont{0.4}
\begin{tikzpicture}[scale=0.55]
\draw (2,-1)--(0,0)--(0,1)--(0.92,1.93);
\draw (0,0)--(1,1)--(1,1.9);
\draw (2,-1)--(2,1)--(1.08,1.93);
\draw [fill] (2,-1) circle [radius=0.1];
\draw [fill] (1,-0.5) circle [radius=0.1];
\draw [fill] (0,0) circle [radius=0.1];
\draw [fill] (0,1) circle [radius=0.1];
\draw [fill] (1,1) circle [radius=0.1];
\draw (1,2) circle [radius=0.1];
\draw [fill] (2,1) circle [radius=0.1];
\node [right] at (2,-1) {$1$};
\node [right] at (1.05,-0.45) {$0$};
\node [right] at (0.05,0.05) {$1$};
\node [right] at (0,1) {\textrm{-}$1$};
\node [right] at (1,1) {\textrm{-}$1$};
\node [right] at (2,1) {\textrm{-}$1$};
\node [right] at (1,2) {$1$};
\end{tikzpicture}}
}
\subfigure[$7_{\rm BG}$]
{{\scalefont{0.4}
\begin{tikzpicture}[scale=0.55]
\draw (2,-1)--(0,0)--(0,1)--(0.92,1.93);
\draw (1,-0.5)--(1,1)--(1,1.9);
\draw (2,-1)--(2,1)--(1.08,1.93);
\draw [fill] (2,-1) circle [radius=0.1];
\draw [fill] (1,-0.5) circle [radius=0.1];
\draw [fill] (0,0) circle [radius=0.1];
\draw [fill] (0,1) circle [radius=0.1];
\draw [fill] (1,1) circle [radius=0.1];
\draw (1,2) circle [radius=0.1];
\draw [fill] (2,1) circle [radius=0.1];
\node [right] at (2,-1) {$1$};
\node [right] at (1,-0.4) {$1$};
\node [right] at (0,0.1) {$0$};
\node [right] at (0,1) {\textrm{-}$1$};
\node [right] at (1,1) {\textrm{-}$1$};
\node [right] at (2,1) {\textrm{-}$1$};
\node [right] at (1,2) {$1$};
\end{tikzpicture}}
}
\subfigure[$7_{\rm BH}$]
{{\scalefont{0.4}
\begin{tikzpicture}[scale=0.55]
\draw (2,-1)--(0,0)--(0,1)--(0.92,1.93);
\draw (0,0)--(1,1)--(1,1.9);
\draw (1,-0.5)--(2,1)--(1.08,1.93);
\draw [fill] (2,-1) circle [radius=0.1];
\draw [fill] (1,-0.5) circle [radius=0.1];
\draw [fill] (0,0) circle [radius=0.1];
\draw [fill] (0,1) circle [radius=0.1];
\draw [fill] (1,1) circle [radius=0.1];
\draw (1,2) circle [radius=0.1];
\draw [fill] (2,1) circle [radius=0.1];
\node [right] at (2,-1) {$0$};
\node [right] at (1.05,-0.45) {$1$};
\node [right] at (0.05,0.05) {$1$};
\node [right] at (0,1) {\textrm{-}$1$};
\node [right] at (1,1) {\textrm{-}$1$};
\node [right] at (2,1) {\textrm{-}$1$};
\node [right] at (1,2) {$1$};
\end{tikzpicture}}
}
\subfigure[$7_{\rm BI}$]
{{\scalefont{0.4}
\begin{tikzpicture}[scale=0.55]
\draw (2,0)--(1,-0.5)--(0,0)--(0,1)--(0.92,1.93);
\draw (0,0)--(1,1)--(1,1.9);
\draw (2,0)--(2,1)--(1.08,1.93);
\draw [fill] (2,0) circle [radius=0.1];
\draw [fill] (1,-0.5) circle [radius=0.1];
\draw [fill] (0,0) circle [radius=0.1];
\draw [fill] (0,1) circle [radius=0.1];
\draw [fill] (1,1) circle [radius=0.1];
\draw (1,2) circle [radius=0.1];
\draw [fill] (2,1) circle [radius=0.1];
\node [right] at (2,0) {$0$};
\node [right] at (1,-0.6) {$1$};
\node [right] at (0.05,0.05) {$1$};
\node [right] at (0,1) {\textrm{-}$1$};
\node [right] at (1,1) {\textrm{-}$1$};
\node [right] at (2,1) {\textrm{-}$1$};
\node [right] at (1,2) {$1$};
\end{tikzpicture}}
}
\centerline{\bf Figure 5.}
\end{figure}
\setcounter{subfigure}{0}
\begin{figure}[ht!]
\centering
\subfigure[$7_{\rm CA}$]
{{\scalefont{0.4}
\begin{tikzpicture}[scale=0.55]
\draw (1,0)--(0.2,1)--(1,2);
\draw (1,0)--(1,2);
\draw (0.2,1)--(0.2,1.9);
\draw (0.2,2)--(0.56,2.92);
\draw (1,2)--(0.64,2.92);
\draw (1,0)--(1.8,1)--(1,2);
\draw [fill] (0.2,1) circle [radius=0.1];
\draw [fill] (1,1) circle [radius=0.1];
\draw [fill] (1.8,1) circle [radius=0.1];
\draw [fill] (1,0) circle [radius=0.1];
\draw [fill] (1,2) circle [radius=0.1];
\draw [fill] (0.2,2) circle [radius=0.1];
\draw (0.6,3) circle [radius=0.1];
\node [right] at (0.6,3) {$1$};
\node [right] at (1,0) {$0$};
\node [right] at (1,2) {\textrm{-}$1$};
\node [right] at (1,1) {$0$};
\node [right] at (1.8,1) {$0$};
\node [right] at (0.2,1) {$1$};
\node [right] at (0.2,2) {\textrm{-}$1$};
\end{tikzpicture}}
}
\subfigure[$7_{\rm CB}$]
{{\scalefont{0.4}
\begin{tikzpicture}[scale=0.55]
\draw (1,0)--(0.3,1)--(1.2,2);
\draw (1,0)--(1,1)--(1.2,2);
\draw (1,0)--(1.8,1)--(1.2,2);
\draw (1,0)--(2.6,1)--(2.32,2.2);
\draw (2.22,2.25)--(1.2,2);
\draw (2.3,2.3) circle [radius=0.1];
\draw [fill] (1,0) circle [radius=0.1];
\draw [fill] (0.3,1) circle [radius=0.1];
\draw [fill] (1,1) circle [radius=0.1];
\draw [fill] (1.8,1) circle [radius=0.1];
\draw [fill] (2.6,1) circle [radius=0.1];
\draw [fill] (1.2,2) circle [radius=0.1];
\draw [fill] (1.8,1) circle [radius=0.1];
\node [right] at (2.3,2.3) {$1$};
\node [right] at (1,0) {$1$};
\node [right] at (0.3,1) {$0$};
\node [right] at (1.8,1) {$0$};
\node [right] at (2.6,1) {\textrm{-}$1$};
\node [right] at (1,1) {$0$};
\node [right] at (1.2,1.9) {\textrm{-}$1$};
\end{tikzpicture}}
}
\subfigure[$7_{\rm CC}$]
{{\scalefont{0.4}
\begin{tikzpicture}[scale=0.55]
\draw (1.5,0)--(0.5,1)--(1.5,2);
\draw (1.5,0)--(1.2,1)--(1.5,2);
\draw (1.5,0)--(2,1)--(2.36,1.92);
\draw (1.5,0)--(2.8,1);
\draw (2.8,1)--(2.44,1.92);
\draw (2,1)--(1.5,2);
\draw (2.4,2) circle [radius=0.1];
\draw [fill] (1.5,0) circle [radius=0.1];
\draw [fill] (0.5,1) circle [radius=0.1];
\draw [fill] (1.2,1) circle [radius=0.1];
\draw [fill] (2,1) circle [radius=0.1];
\draw [fill] (2.8,1) circle [radius=0.1];
\draw [fill] (1.5,2) circle [radius=0.1];
\draw [fill] (2,1) circle [radius=0.1];
\node [right] at (2.4,2) {$1$};
\node [right] at (1.5,-0.05) {$1$};
\node [right] at (0.5,1) {$0$};
\node [right] at (2,1) {\textrm{-}$1$};
\node [right] at (2.8,1) {\textrm{-}$1$};
\node [right] at (1.2,1) {$0$};
\node [right] at (1.5,2) {$0$};
\end{tikzpicture}}
}
\subfigure[$7_{\rm CD}$]
{{\scalefont{0.4}
\begin{tikzpicture}[scale=0.55]
\draw (1,0)--(0.25,1)--(1,2)--(1.65,2.39);
\draw (1,0)--(1.75,1)--(1,2);
\draw (1,0)--(2.5,1)--(1.74,2.3);
\draw (1,0)--(3.25,1)--(1.79,2.32);
\draw [fill] (0.25,1) circle [radius=0.1];
\draw [fill] (1.75,1) circle [radius=0.1];
\draw [fill] (1,0) circle [radius=0.1];
\draw [fill] (1,2) circle [radius=0.1];
\draw (1.75,2.4) circle [radius=0.1];
\draw [fill] (2.5,1) circle [radius=0.1];
\draw [fill] (3.25,1) circle [radius=0.1];
\node [right] at (0.25,1.05) {$0$};
\node [right] at (1.75,1.05) {$0$};
\node [right] at (1,1.95) {\textrm{-}$1$};
\node [right] at (1.75,2.4) {$1$};
\node [right] at (2.5,1.05) {\textrm{-}$1$};
\node [right] at (3.25,1.05) {\textrm{-}$1$};
\node [right] at (1,-0.1) {$2$};
\end{tikzpicture}}
}
\subfigure[$7_{\rm CE}$]
{{\scalefont{0.4}
\begin{tikzpicture}[scale=0.55]
\draw (1,0)--(0.25,1)--(1,2)--(1.68,2.93);
\draw (1,0)--(1.75,1)--(1,2);
\draw (1.75,1)--(1.75,2.9);
\draw (1.75,1)--(2.5,2)--(1.82,2.93);
\draw [fill] (0.25,1) circle [radius=0.1];
\draw [fill] (1.75,1) circle [radius=0.1];
\draw [fill] (1,0) circle [radius=0.1];
\draw [fill] (1.75,2) circle [radius=0.1];
\draw [fill] (2.5,2) circle [radius=0.1];
\draw (1.75,3) circle [radius=0.1];
\draw [fill] (1,2) circle [radius=0.1];
\node [right] at (1,0) {$0$};
\node [right] at (0.25,1) {$0$};
\node [right] at (1.75,1) {$2$};
\node [right] at (1,2) {\textrm{-}$1$};
\node [right] at (1.75,2) {\textrm{-}$1$};
\node [right] at (2.5,2) {\textrm{-}$1$};
\node [right] at (1.75,3) {$1$};
\end{tikzpicture}}
}
\centerline{\bf Figure 6.}
\end{figure}
\begin{figure}[ht!]
\centering
\subfigure[$7_{\rm DA}$]
{{\scalefont{0.4}
\begin{tikzpicture}[scale=0.55]
\draw (2.5,-0.5)--(1.59,-0.22);
\draw (2.5,-0.5)--(0.5,1)--(1.43,1.93);
\draw (2.5,-0.5)--(1.2,1)--(1.43,1.93);
\draw (2.5,-0.5)--(2,1)--(1.57,1.93);
\draw (2.5,-0.5)--(2.8,1)--(1.57,1.93);
\draw [fill] (2.5,-0.5) circle [radius=0.1];
\draw (1.5,-0.2) circle [radius=0.1];
\draw [fill] (0.5,1) circle [radius=0.1];
\draw [fill] (1.2,1) circle [radius=0.1];
\draw [fill] (2,1) circle [radius=0.1];
\draw [fill] (2.8,1) circle [radius=0.1];
\draw [fill] (1.5,2) circle [radius=0.1];
\draw [fill] (2,1) circle [radius=0.1];
\node [right] at (2.5,-0.5) {\textrm{-}$1$};
\node [right] at (1.44,-0.45) {$1$};
\node [right] at (0.5,1) {$0$};
\node [right] at (2,1) {$0$};
\node [right] at (2.8,1) {$0$};
\node [right] at (1.2,1) {$0$};
\node [right] at (1.5,2) {$0$};
\end{tikzpicture}}
}
\subfigure[$7_{\rm DB}$]
{{\scalefont{0.4}
\begin{tikzpicture}[scale=0.55]
\draw (1.5,0)--(0.5,1)--(1.5,2);
\draw (1.5,0)--(1.2,1)--(1.43,1.93);
\draw (1.5,0)--(2,1)--(1.57,1.93);
\draw (1.5,0)--(2.8,1)--(1.57,1.93);
\draw (2.8,1)--(2.8,1.9);
\draw (2.8,2) circle [radius=0.1];
\draw [fill] (1.5,0) circle [radius=0.1];
\draw [fill] (0.5,1) circle [radius=0.1];
\draw [fill] (1.2,1) circle [radius=0.1];
\draw [fill] (2,1) circle [radius=0.1];
\draw [fill] (2.8,1) circle [radius=0.1];
\draw [fill] (1.5,2) circle [radius=0.1];
\draw [fill] (2,1) circle [radius=0.1];
\node [right] at (2.8,2) {$1$};
\node [right] at (1.5,0.02) {\ $0$};
\node [right] at (0.5,1) {$0$};
\node [right] at (2,1) {$0$};
\node [right] at (2.8,1) {\textrm{-}$1$};
\node [right] at (1.2,1) {$0$};
\node [right] at (1.5,2) {$0$};
\end{tikzpicture}}
}
\subfigure[$7_{\rm DC}$]
{{\scalefont{0.4}
\begin{tikzpicture}[scale=0.55]
\draw (1.5,0)--(0.5,1)--(1.5,2)--(2.43,2.33);
\draw (1.5,0)--(1.2,1)--(1.43,1.93);
\draw (1.5,0)--(2,1)--(1.57,1.93);
\draw (1.5,0)--(2.8,1)--(1.57,1.93);
\draw (2.5,2.4) circle [radius=0.1];
\draw [fill] (1.5,0) circle [radius=0.1];
\draw [fill] (0.5,1) circle [radius=0.1];
\draw [fill] (1.2,1) circle [radius=0.1];
\draw [fill] (2,1) circle [radius=0.1];
\draw [fill] (2.8,1) circle [radius=0.1];
\draw [fill] (1.5,2) circle [radius=0.1];
\draw [fill] (2,1) circle [radius=0.1];
\node [right] at (2.5,2.4) {$1$};
\node [right] at (1.5,0.02) {\ $0$};
\node [right] at (0.5,1) {$0$};
\node [right] at (2,1) {$0$};
\node [right] at (2.8,1) {$0$};
\node [right] at (1.2,1) {$0$};
\node [right] at (1.6,1.95) {\textrm{-}$1$};
\end{tikzpicture}}
}
\subfigure[$7_{\rm DD}$]
{{\scalefont{0.4}
\begin{tikzpicture}[scale=0.55]
\draw (2.5,-0.5)--(1.5,0)--(0.5,1)--(1.43,1.93);
\draw (1.5,0)--(1.2,1)--(1.43,1.93);
\draw (1.5,0)--(2,1)--(1.57,1.93);
\draw (1.5,0)--(2.8,1)--(1.57,1.93);
\draw [fill] (2.5,-0.5) circle [radius=0.1];
\draw [fill] (1.5,0) circle [radius=0.1];
\draw [fill] (0.5,1) circle [radius=0.1];
\draw [fill] (1.2,1) circle [radius=0.1];
\draw [fill] (2,1) circle [radius=0.1];
\draw [fill] (2.8,1) circle [radius=0.1];
\draw (1.5,2) circle [radius=0.1];
\draw [fill] (2,1) circle [radius=0.1];
\node [right] at (2.5,-0.5) {$0$};
\node [right] at (1.5,0.02) {\ $3$};
\node [right] at (0.5,1) {\textrm{-}$1$};
\node [right] at (2,1) {\textrm{-}$1$};
\node [right] at (2.8,1) {\textrm{-}$1$};
\node [right] at (1.2,1) {\textrm{-}$1$};
\node [right] at (1.5,2) {$1$};
\end{tikzpicture}}
}
\subfigure[$7_{\rm DE}$]
{{\scalefont{0.4}
\begin{tikzpicture}[scale=0.55]
\draw (2.5,-0.5)--(1.5,0)--(0.5,1)--(1.43,1.93);
\draw (2.5,-0.5)--(1.2,1)--(1.43,1.93);
\draw (2.5,-0.5)--(2,1)--(1.57,1.93);
\draw (2.5,-0.5)--(2.8,1)--(1.57,1.93);
\draw [fill] (2.5,-0.5) circle [radius=0.1];
\draw [fill] (1.5,0) circle [radius=0.1];
\draw [fill] (0.5,1) circle [radius=0.1];
\draw [fill] (1.2,1) circle [radius=0.1];
\draw [fill] (2,1) circle [radius=0.1];
\draw [fill] (2.8,1) circle [radius=0.1];
\draw (1.5,2) circle [radius=0.1];
\draw [fill] (2,1) circle [radius=0.1];
\node [right] at (2.5,-0.5) {$3$};
\node [right] at (1.5,0.05) {$0$};
\node [right] at (0.5,1) {\textrm{-}$1$};
\node [right] at (2,1) {\textrm{-}$1$};
\node [right] at (2.8,1) {\textrm{-}$1$};
\node [right] at (1.2,1) {\textrm{-}$1$};
\node [right] at (1.5,2) {$1$};
\end{tikzpicture}}
}
\centerline{\bf Figure 7.}
\end{figure}
\begin{figure}[ht!]
\centering
\setcounter{subfigure}{0}
\subfigure[$7_{\rm E}$]
{{\scalefont{0.4}
\begin{tikzpicture}[scale=0.55]
\draw (1,0)--(0.25,1)--(0.92,1.92);
\draw (1,0)--(1,-0.9);
\draw (1,0)--(1.75,1)--(1.08,1.92);
\draw (1,-1)--(1.75,0)--(1.75,1);
\draw (1,-1)--(-0.5,1)--(0.9,1.95);
\draw [fill] (-0.5,1) circle [radius=0.1];
\draw [fill] (0.25,1) circle [radius=0.1];
\draw [fill] (1.75,1) circle [radius=0.1];
\draw [fill] (1,0) circle [radius=0.1];
\draw (1,2) circle [radius=0.1];
\draw [fill] (1.75,0) circle [radius=0.1];
\draw [fill] (1,-1) circle [radius=0.1];
\node [right] at (1,0) {$1$};
\node [right] at (0.25,1) {\textrm{-}$1$};
\node [right] at (1.75,1) {\textrm{-}$1$};
\node [right] at (1,2) {$1$};
\node [right] at (1,-1) {$1$};
\node [right] at (-0.5,1) {\textrm{-}$1$};
\node [right] at (1.75,0) {$0$};
\end{tikzpicture}}
}
\subfigure[$7_{\rm F}$]
{{\scalefont{0.4}
\begin{tikzpicture}[scale=0.55]
\draw (2,0)--(1,-0.5)--(0,0)--(0,1)--(0.92,1.93);
\draw (2,0)--(1,1);
\draw (0,0)--(1,1)--(1,1.9);
\draw (2,0)--(2,1)--(1.08,1.93);
\draw [fill] (2,0) circle [radius=0.1];
\draw [fill] (1,-0.5) circle [radius=0.1];
\draw [fill] (0,0) circle [radius=0.1];
\draw [fill] (0,1) circle [radius=0.1];
\draw [fill] (1,1) circle [radius=0.1];
\draw (1,2) circle [radius=0.1];
\draw [fill] (2,1) circle [radius=0.1];
\node [right] at (2,0) {$1$};
\node [right] at (1,-0.6) {$0$};
\node [right] at (0.05,0.05) {$1$};
\node [right] at (0,1) {\textrm{-}$1$};
\node [right] at (1,1) {\textrm{-}$1$};
\node [right] at (2,1) {\textrm{-}$1$};
\node [right] at (1,2) {$1$};
\end{tikzpicture}}
}
\subfigure[$7_{\rm G}$]
{{\scalefont{0.4}
\begin{tikzpicture}[scale=0.55]
\draw (2.3,-0.5)--(1,0)--(0.3,1)--(1.43,1.93);
\draw (1,0)--(1,1)--(1.43,1.93);
\draw (1,0)--(1.8,1)--(1.57,1.93);
\draw (2.3,-0.5)--(2.6,1)--(1.57,1.93);
\draw [fill] (2.3,-0.5) circle [radius=0.1];
\draw [fill] (1,0) circle [radius=0.1];
\draw [fill] (0.3,1) circle [radius=0.1];
\draw [fill] (1,1) circle [radius=0.1];
\draw [fill] (1.8,1) circle [radius=0.1];
\draw [fill] (2.6,1) circle [radius=0.1];
\draw (1.5,2) circle [radius=0.1];
\draw [fill] (1.8,1) circle [radius=0.1];
\node [right] at (2.3,-0.5) {$1$};
\node [right] at (0.9,-0.3) {$2$};
\node [right] at (0.3,1) {\textrm{-}$1$};
\node [right] at (1.8,1) {\textrm{-}$1$};
\node [right] at (2.6,1) {\textrm{-}$1$};
\node [right] at (1,1) {\textrm{-}$1$};
\node [right] at (1.5,2) {$1$};
\end{tikzpicture}}
}
\subfigure[$7_{\rm H}$]
{{\scalefont{0.4}
\begin{tikzpicture}[scale=0.55]
\draw (2.5,-0.5)--(1,0)--(0.3,1)--(1.43,1.93);
\draw (1,0)--(1,1)--(1.43,1.93);
\draw (2.5,-0.5)--(2,1)--(1.57,1.93);
\draw (2.5,-0.5)--(2.8,1)--(1.57,1.93);
\draw [fill] (2.5,-0.5) circle [radius=0.1];
\draw [fill] (1,0) circle [radius=0.1];
\draw [fill] (0.3,1) circle [radius=0.1];
\draw [fill] (1,1) circle [radius=0.1];
\draw [fill] (2,1) circle [radius=0.1];
\draw [fill] (2.8,1) circle [radius=0.1];
\draw (1.5,2) circle [radius=0.1];
\draw [fill] (2,1) circle [radius=0.1];
\node [right] at (2.5,-0.5) {$2$};
\node [right] at (0.9,-0.3) {$1$};
\node [right] at (0.3,1) {\textrm{-}$1$};
\node [right] at (2,1) {\textrm{-}$1$};
\node [right] at (2.8,1) {\textrm{-}$1$};
\node [right] at (1,1) {\textrm{-}$1$};
\node [right] at (1.5,2) {$1$};
\end{tikzpicture}}
}
\subfigure[$7_{\rm I}$]
{{\scalefont{0.4}
\begin{tikzpicture}[scale=0.55]
\draw (1,0)--(-0.6,1)--(0.92,1.92);
\draw (1,0)--(0.2,1)--(0.92,1.92);
\draw (1,0)--(1,1.9);
\draw (1,0)--(1.8,1)--(1.08,1.92);
\draw (1,0)--(2.6,1)--(1.08,1.92);
\draw [fill] (0.2,1) circle [radius=0.1];
\draw [fill] (1,1) circle [radius=0.1];
\draw [fill] (1.8,1) circle [radius=0.1];
\draw [fill] (1,0) circle [radius=0.1];
\draw [fill] (2.6,1) circle [radius=0.1];
\draw [fill] (-0.6,1) circle [radius=0.1];
\draw (1,2) circle [radius=0.1];
\node [right] at (1,-0.1) {$4$};
\node [right] at (0.2,1) {\textrm{-}$1$};
\node [right] at (1,1) {\textrm{-}$1$};
\node [right] at (-0.6,1) {\textrm{-}$1$};
\node [right] at (2.6,1) {\textrm{-}$1$};
\node [right] at (1.8,1) {\textrm{-}$1$};
\node [right] at (1,2) {$1$};
\end{tikzpicture}}
}
\centerline{\bf Figure 8.}
\end{figure}

\begin{proof}
(i) This case is trivial, since if $S$ can be constructed by (M$_{1,i}$) and (M$_{2,i}$) only, then $f\in{\cal F}_2$ is a sufficient condition for the invertibility of $[S]_f$. (ii) Let $S\in 7_{\rm AA},7_{\rm AB},\ldots,7_{\rm AX}$. Then $S$ can be constructed by applying (M$_{1,i}$) six times and (M$_{3,i}$) once. Due to the zeros of the M\"obius function, the condition $f\in{\cal G}_{3, 5}$ guarantees the invertibility of $[S_i]_f$ when (M$_{3,i}$) is applied. Everytime when (M$_{1,i}$) is applied the invertibility follows from the condition $f\in{\cal F}_1$. (iii) The situation is  similar to the cases $S\in 7_{\rm BA},7_{\rm BB},\ldots,7_{\rm BI}$. The only difference is that the assumption $f\in {\cal G}_{3, 6}$ implies the invertibility of $[S_i]_f$ when (M$_{3,i}$) is used. (iv) In the cases $S\in 7_{\rm CA},7_{\rm CB},7_{\rm CC},7_{\rm CD},7_{\rm CE}$ the methods (M$_{1,i}$), (M$_{2,i}$) and (M$_{3,i}$) are all needed in the construction of the set $S$. In order to the matrix $[S]_f$ to be invertible, these methods require the assumptions $f\in{\cal F}_1$, $f\in{\cal F}_2$ and $f\in{\cal G}_{3, 5}$, respectively. (v) If $S\in 7_{\rm DA},7_{\rm DB},7_{\rm DC},7_{\rm DD},7_{\rm DE}$, then (M$_{1,i}$) and (M$_{4,i}$) are the only used methods. Here $f\in {\cal G}_{4, 6}$ assures the invertibility of $[S_i]_f$ when (M$_{4,i}$) is applied, otherwise the invertibility of $[S_i]_f$ follows from the condition $f\in {\cal F}_1$. (vi) The case $S\in 7_{\rm E}$ has much recemblance to the case (iv); here we just need the condition $f\in {\cal G}_{3, 6}$ instead of $f\in {\cal G}_{3, 5}$ when the method (M$_{3,i}$) is used. (viii)-(ix) In the cases $S\in 7_{\rm G}$ and $S\in 7_{\rm H}$ the set $S_{n-1}$ can be constructed by (M$_{1,i}$) only and $S$ can be constructed by (M$_{4,7}$) from $S_{n-1}$. Therefore in both cases the assumption $f\in{\cal F}_1$ guarantees that the matrix $[S_{n-1}]_f$ is invertible, whereas either condition $f\in {\cal G}_{4, 7}^{(1)}$ or $f\in{\cal G}_{4, 7}^{(2)}$ is needed to assure the invertibility of $[S]_f$ when the last element is added. (x) The case $S\in 7_{\rm I}$ is similar, here only the method (M$_{5,7}$) is used instead of (M$_{4,7}$) and the condition $f\in {\cal G}_{5, 7}$ is needed instead of assuming $f\in {\cal G}_{4, 7}^{(1)}$ or $f\in{\cal G}_{4, 7}^{(2)}$. This last result also follows from Theorem \ref{th:3.4}.
\end{proof}

\begin{corollary}\label{cor:4.5}
If $S$ is a meet closed set with $7$ elements and $f\in{\cal F}_2\cap {\cal G}_{3, 5}\cap {\cal G}_{3, 6}\cap {\cal G}_{4, 6}\cap {\cal G}_{4, 7}^{(1)}\cap {\cal G}_{4, 7}^{(2)}\cap {\cal G}_{5, 7}$, then $[S]_f$ is invertible. In particular, if $S$ is a gcd-closed set with $7$ elements, then $[S]$ is invertible.
\end{corollary}

\begin{proof}
The first part of this corollary is obvious, since ${\cal F}_1 \supseteq {\cal F}_2$ and the sets in parts (iii), (v) and (vi) respectively belong to classes ${\cal S}_{3, 7}$, ${\cal S}_{4, 7}^{(1)}$ and ${\cal S}_{4, 7}^{(2)}$. We prove the second part of this corollary. Since by Remark \ref{rem:B-L}, Corollary \ref{cor:3.2} and Corollary \ref{cor:4.4} $N\in{\cal F}_2\cap {\cal G}_{3, 5}\cap{\cal G}_{3, 6}\cap {\cal G}_{4, 6}\cap {\cal G}_{5, 7}$, it suffices to prove that $N\in {\cal G}_{3, 7}\cap {\cal G}_{4, 7}^{(1)}\cap {\cal G}_{4, 7}^{(2)}$. We prove first that $N\in {\cal G}_{3, 7}$ $(S\in 7_{\rm F})$. Let $x_1=\hbox{gcd}(x_2,x_3)=\hbox{gcd}(x_3,x_4)=\hbox{gcd}(x_2,x_6)=\hbox{gcd}(x_4,x_6)$, $x_2=\hbox{gcd}(x_4,x_5)$, $x_3=\hbox{gcd}(x_5,x_6)$ and $\hbox{lcm}(x_4,x_5,x_6)\mid x_7$. Thus $x_2=ax_1$, $x_3=bx_1$, $x_4=acx_1$, $x_5=abdx_1$, $x_6=bex_1$, where $a,b,c,e\geq2$ and $d\geq1$. Since $\gcd(c,bd)=1$, either $c\geq3$ or $b,d\geq3$ and we have $(bd-1)(c-1)-1>0$. In addition, $e-1>0$ and $x_1,x_7>0$ and thus we obtain
\begin{eqnarray}\label{eq:4.11}
&&\frac{1}{x_7}-\frac{1}{x_6}-\frac{1}{x_5}-\frac{1}{x_4}+\frac{1}{x_3}+\frac{1}{x_2}= \\
&&\qquad\qquad=\;\frac{1}{x_7}+\frac{-acd-ce-bde+acde+bcde}{abcdex_1} \nonumber \\
&&\qquad\qquad=\;\frac{1}{x_7}+\frac{acd(e-1)+e[(bd-1)(c-1)-1]}{abcdex_1}>0.\nonumber
\end{eqnarray}
Thus $N\in {\cal G}_{3, 7}$.

We prove second that $N\in {\cal G}_{4, 7}^{(1)}$ $(S\in 7_{\rm G})$. Let $x_1=\hbox{gcd}(x_2,x_3)=\hbox{gcd}(x_4,x_3)$ $=\hbox{gcd}(x_5,x_3)=\hbox{gcd}(x_6,x_3)$, $x_2=\hbox{gcd}(x_4,x_5)=\hbox{gcd}(x_4,x_6)=\hbox{gcd}(x_5,x_6)$ and $\hbox{lcm}(x_3,x_4,x_5,x_6)\mid x_7$. Thus $x_2=ax_1$, $x_3=bx_1$, $x_4=acx_1$, $x_5=adx_1$, $x_6=aex_1$, where $a,b,c,d,e\geq2$. Here $\gcd(d,e)=1$, which implies that $d\neq e$ and either $d>2$ or $e>2$. Therefore $de-d-e>0$, and since also $b-1,c-1>0$ and $x_1,x_7>0$, we have
\begin{eqnarray}\label{eq:4.12}
&&\frac{1}{x_7}-\frac{1}{x_6}-\frac{1}{x_5}-\frac{1}{x_4}-\frac{1}{x_3}+\frac{2}{x_2}+\frac{1}{x_1}= \\
&&\qquad\qquad=\;\frac{1}{x_7}+\frac{-bcd-bce-bde-acde+2bcde+abcde}{abcdex_1} \nonumber \\
&&\qquad\qquad=\;\frac{1}{x_7}+\frac{bc(de-d-e)+bde(c-1)+acde(b-1)}{abcdex_1}>0.\nonumber
\end{eqnarray}
Thus $N\in {\cal G}_{4, 7}^{(1)}$.

We prove third that $N\in {\cal G}_{4, 7}^{(2)}$ $(S\in7_{\rm H})$. Let $x_1\mid x_2$, $x_1=\hbox{gcd}(x_3,x_4)=\hbox{gcd}(x_3,x_5)=\hbox{gcd}(x_4,x_5)=\hbox{gcd}(x_3,x_6)=\hbox{gcd}(x_4,x_6)$, $x_2=\hbox{gcd}(x_5,x_6)$ and $\hbox{lcm}(x_3,x_4,x_5,x_6)\mid x_7$. Thus $x_2=ax_1$, $x_3=bx_1$, $x_4=cx_1$, $x_5=adx_1$, $x_6=aex_1$, where $a,b,c,d,e\geq2$. Since $\gcd(e,d)=1$, we have either $d>2$ or $e>2$ and further $de-d-e>0$. In addition, since $b-1,c-1>0$ and $x_1,x_7>0$ we have
\begin{eqnarray}\label{eq:4.13}
&&\frac{1}{x_7}-\frac{1}{x_6}-\frac{1}{x_5}-\frac{1}{x_4}-\frac{1}{x_3}+\frac{1}{x_2}+\frac{2}{x_1}= \\
&&\qquad\qquad=\;\frac{1}{x_7}+\frac{-bcd-bce-abde-acde+bcde+2abcde}{abcdex_1} \nonumber \\
&&\qquad\qquad=\;\frac{1}{x_7}+\frac{bc(de-d-e)+abde(c-1)+acde(b-1)}{abcdex_1}>0.\nonumber
\end{eqnarray}
Thus $N\in {\cal G}_{4, 7}^{(2)}$.
\end{proof}

\subsection{Cases \texorpdfstring{$n=8,9,\ldots$}{n=8,9,...}}\label{subsect:4.8}

Haukkanen, Wang and Sillanp\"a\"a \cite{HauWanSil} showed that the Bourque-Ligh conjecture is false by giving the counterexample \[S=\{1, 2, 3, 4, 5, 6, 10, 45, 180\},\] where $n=9$. Hong \cite{Hong99} solved the conjecture completely (in a sense) showing that it holds for $n\leq7$ and does not hold generally for $n\geq8$. The counterexample given by Hong is
\begin{align*}
S&=\{1,2,3,5,36,230,825,227700\}\\&=\{1,2,3,5,6(2\cdot3),23(2\cdot5),55(3\cdot5),(6\cdot23\cdot55)(2\cdot3\cdot5)\}.
\end{align*}
For this counterexample given by Hong \cite{Hong99} we have $S\in {\cal S}_{3, 8}$ with $[S]$ being singular. Thus $N\not\in {\cal G}_{3, 8}$ and, more general, $N\not\in{\cal F}_3$. For any $n\ge 8$ we are also able to construct a gcd-closed set $S$ possessing the structure given on the left side of Figure 9 as a subsemilattice, which makes the LCM matrix $[S]$ singular. These counterexamples together with Corollaries \ref{cor:4.1}--\ref{cor:4.5} serve as a lattice-theoretic solution of the Bourque-Ligh conjecture.

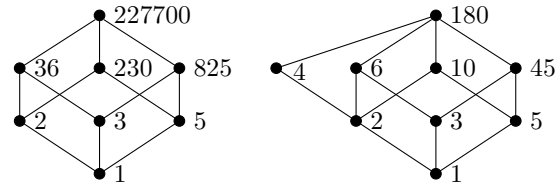
\begin{figure}[htb!]
\centering
\subfigure
{{\scalefont{0.8}
\begin{tikzpicture}[scale=0.7]
\draw (1,0)--(-0.5,1)--(-0.5,2)--(0.92,2.92);
\draw (1,0)--(1,1)--(-0.5,2);
\draw (-0.5,1)--(1,2)--(1,2.9);
\draw (1,0)--(2.5,1)--(2.5,2)--(1.08,2.92);
\draw (1,1)--(2.5,2);
\draw (2.5,1)--(1,2);
\draw [fill] (-0.5,1) circle [radius=0.1];
\draw [fill] (1,1) circle [radius=0.1];
\draw [fill] (2.5,1) circle [radius=0.1];
\draw [fill] (1,0) circle [radius=0.1];
\draw [fill] (-0.5,2) circle [radius=0.1];
\draw [fill] (1,2) circle [radius=0.1];
\draw [fill] (2.5,2) circle [radius=0.1];
\draw [fill] (1,3) circle [radius=0.1];
\node [right] at (1.1,0) {$1$};
\node [right] at (-0.4,1) {$2$};
\node [right] at (1.1,1) {$3$};
\node [right] at (2.6,1) {$5$};
\node [right] at (-0.4,2) {$36$};
\node [right] at (1.1,2) {$230$};
\node [right] at (2.6,2) {$825$};
\node [right] at (1.1,3) {$227700$};
\end{tikzpicture}}
}
\subfigure
{{\scalefont{0.8}
\begin{tikzpicture}[scale=0.7]
\draw (1,0)--(-0.5,1)--(-0.5,2)--(0.92,2.92);
\draw (1,0)--(1,1)--(-0.5,2);
\draw (-0.5,1)--(1,2)--(1,2.9);
\draw (1,0)--(2.5,1)--(2.5,2)--(1.08,2.92);
\draw (1,1)--(2.5,2);
\draw (-0.5,1)--(-2,2)--(1,3);
\draw (2.5,1)--(1,2);
\draw [fill] (-0.5,1) circle [radius=0.1];
\draw [fill] (1,1) circle [radius=0.1];
\draw [fill] (2.5,1) circle [radius=0.1];
\draw [fill] (1,0) circle [radius=0.1];
\draw [fill] (-0.5,2) circle [radius=0.1];
\draw [fill] (1,2) circle [radius=0.1];
\draw [fill] (2.5,2) circle [radius=0.1];
\draw [fill] (-2,2) circle [radius=0.1];
\draw [fill] (1,3) circle [radius=0.1];
\node [right] at (1.1,0) {$1$};
\node [right] at (-0.4,1) {$2$};
\node [right] at (1.1,1) {$3$};
\node [right] at (2.6,1) {$5$};
\node [right] at (-0.4,2) {$6$};
\node [right] at (1.1,2) {$10$};
\node [right] at (2.6,2) {$45$};
\node [right] at (-1.85,1.92) {$4$};
\node [right] at (1.1,3) {$180$};
\end{tikzpicture}}
}
\caption*{{\bf Figure 9.} On the left is the counterexample for the Bourque-Ligh conjecture given by Hong. The lattice on the right is the counterexample by Haukkanen, Wang and Sillanp\"a\"a.}
\end{figure}

\noindent
{\bf Acknowledgements} We wish to thank Jori M\"antysalo for valuable help with Sage. We also wish to thank the anonymous referee, who a couple of years ago gave many useful comments and suggestions regarding an earlier version of this article.


\section*{References}

\begin{thebibliography}{99}

\bibitem{Aig} M.~Aigner, {\it Combinatorial Theory}, Springer--Verlag, Berlin--New York, 1979.

\bibitem{AltSagTug} E.~Altinisik, B.~E.~Sagan and N.~Tuglu, GCD matrices, posets, and nonintersecting paths, {\it Linear Multilinear Algebra} 53: 75--84 (2005).

\bibitem{AltTugHau} E.~Altinisik, N.~Tuglu and P.~Haukkanen, Determinant and inverse of meet and join matrices, {\it Int. J. Math. Math. Sci.}  Article
ID 37580 (2007).

\bibitem{Bir} G.\ Birkhoff, {\it Lattice Theory}, American Mathematical Society Colloquium Publications, Vol.\ 25, Rhode Island, 1984.

\bibitem{Bour92} K.~Bourque and S.~Ligh, On GCD and LCM matrices, {\it Linear Algebra Appl.} 174: 65--74 (1992).

\bibitem{Hau96} P.~Haukkanen, On meet matrices on posets, {\it Linear Algebra Appl.} 249: 111--123 (1996).

\bibitem{HauIlNalSil} P.~Haukkanen, P.~Ilmonen, A.~Nalli and J.~Sillanp\"a\"a, On unitary analogs of GCD reciprocal LCM matrices, {\it Linear Multilinear Algebra} 58: 599--616 (2010). 

\bibitem{HauSil96} P.~Haukkanen and J.\ Sillanp\"a\"a, Some analogues of Smith's determinant, {\it Linear Multilinear Algebra} 41: 233--244 (1996).

\bibitem{HauWanSil} P.~Haukkanen, J.\ Wang and J. Sillanp\"a\"a, On Smith's determinant, {\it Linear Algebra Appl.} 258: 251--269 (1997).

\bibitem{HeiRei} J. Heitzig and J. Reinhold, Counting finite lattices, {\it Algebra Universalis} 48:  43--53 (2002).

\bibitem{Hong99} S.~Hong, On the Bourque-Ligh conjecture of least common multiple matrices. {\it J. Algebra} 218: 216--228 (1999).

\bibitem{Hong281} S.~Hong, Nonsingularity of matrices associated with classes of arithmetical functions, {\it J.~Algebra} 281: 1--14 (2004).

\bibitem{HongRfold} S.~Hong, Nonsingularity of least common multiple matrices on gcd-closed sets, {\it J.~Number Theory} 113: 1--9 (2005).

\bibitem{HongSun} S.~Hong and Q.~Sun, Determinants of matrices associated with incidence functions on posets, {\it Czechoslovak Math.~J.} 54(129): 431--443 (2004).

\bibitem{Kor8} I.~Korkee, On meet and join matrices associated with incidence functions. Ph.D. thesis, {\it Acta Universitatis Tamperensis} 1149, Tampere University Press, Tampere, 2006.

\bibitem{KorHau1} I.~Korkee and P.~Haukkanen, Bounds for determinants of meet matrices associated with incidence functions, {\it Linear Algebra Appl.} 329: 77--88 (2001).

\bibitem{KorHau2} I.~Korkee and P.~Haukkanen, On meet and join matrices associated with incidence functions, {\it Linear Algebra Appl.} 372: 127--153 (2003).

\bibitem{KorHau6} I.~Korkee and P.~Haukkanen, On meet matrices with respect to reduced, extended and exchanged sets, {\it JP J.~Algebra Number Theory Appl.} 4: 559--575 (2004).

\bibitem{Li} M.~Li, Notes on Hong's conjectures of real number power LCM matrices, {\it J. Algebra} 315: 654--664 (2007).

\bibitem{MatHau1} M.~Mattila and P.~Haukkanen, Determinant and inverse of join matrices on two sets, {\it Linear Algebra Appl.} 438: 3891--3904 (2013).

\bibitem{MatHau2} M.~Mattila and P.~Haukkanen, On positive definiteness and eigenvalues of meet and join matrices, to appear in {\it Discrete Mathematics}. 

\bibitem{McC} P.~J.~McCarthy, {\it Introduction to Arithmetical Functions}, Universitext, Springer--Verlag, New York, 1986.

\bibitem{Raja} B.~V.~Rajarama Bhat, On greatest common divisor matrices and their applications, {\it Linear Algebra Appl.} 158: 77--97 (1991).

\bibitem{Rea} D.~Rearick, Semi-multiplicative functions, {\it Duke Math.~J.} 33: 49--53 (1966).

\bibitem{Sandor} J.~S\'andor and B.~Crstici, {\it Handbook of Number Theory II}, Kluwer (2004).

\bibitem{Siva} R.~Sivaramakrishnan, {\it Classical theory of arithmetic functions}, Monographs and Textbooks in Pure and Applied Mathematics, Vol.~126, Marcel Dekker, Inc., New York (1989).

\bibitem{Smi} H.~J.~S.~Smith, On the value of a certain arithmetical determinant, {\it Proc.\ London Math.\ Soc.} 7, 208--212 (1875/76).

\bibitem{Stan} R.~P.~Stanley, {\it Enumerative combinatorics.} Vol. 1. (Corrected reprint of the 1986 original.) Cambridge Studies in Advanced Mathematics, 49. Cambridge University Press, Cambridge, 1997.

\end{thebibliography}
\end{document}